\documentclass[10pt]{amsart}

\usepackage{amsthm,amsmath,amssymb,amsfonts}
\usepackage{xcolor}
\usepackage{physics}
\usepackage{bbm}
\usepackage[a4paper,margin=3cm]{geometry}
\usepackage[colorlinks=true,citecolor=red,linkcolor=blue]{hyperref}
\usepackage{fancyhdr}
\usepackage{tikz}
\usepackage{caption}

\numberwithin{equation}{section}

\theoremstyle{plain}
\newtheorem{theorem}{Theorem}[section]

\newtheorem{proposition}[theorem]{Proposition}
\newtheorem{lemma}[theorem]{Lemma}

\newtheorem{corollary}[theorem]{Corollary}

\theoremstyle{definition}
\newtheorem{definition}[theorem]{Definition}

\theoremstyle{remark}
\newtheorem{remark}[theorem]{Remark}

\newcommand{\thmref}[1]{Theorem~\ref{#1}}
\newcommand{\secref}[1]{Section~\ref{#1}}
\newcommand{\lemref}[1]{Lemma~\ref{#1}}
\newcommand{\defref}[1]{Definition~\ref{#1}}
\newcommand{\propref}[1]{Proposition~\ref{#1}}

\newcommand{\corref}[1]{Corollary~\ref{#1}}

\newcommand{\R}{\mathbb{R}}

\newcommand{\F}{\mathcal{F}}
\newcommand{\B}{\mathcal{B}}

\newcommand{\N}{\mathbb{N}}
\newcommand{\M}{\mathcal{M}}
\newcommand{\E}{\mathcal{E}}
\newcommand{\Z}{\mathbb{Z}}
\newcommand{\e}{\mathrm{e}}
\newcommand{\ind}{\mathbbm{1}}
\newcommand{\pres}{\mathrm{P}}
\newcommand{\prob}{\mathbb{P}}
\newcommand{\tp}{\mathrm{top}}
\newcommand{\id}{\mathrm{id}}

\newcommand{\probinte}{\dd\prob(\omega)}
\newcommand{\supnorm}[1]{\Vert #1 \Vert_{\infty}}
\newcommand{\onenorm}[1]{\Vert #1 \Vert_{1}}
\newcommand{\almost}{\prob\text{-a.e. }\omega\in\Omega}

\DeclareMathOperator{\diam}{\mathrm{diam}}
\DeclareMathOperator*{\esssup}{\mathrm{ess}\sup}
\DeclareMathOperator*{\essinf}{\mathrm{ess}\inf}

\begin{document}

\title{Induced fiber topological pressure for \break random dynamical systems}
\author{Cunyi Nan}
\date{\today}
\address{Graduate School of Mathematics, Nagoya University,
Furocho, Chikusa-ku, Nagoya, 464-8602, Japan}
\email{cunyi.nan.h5@math.nagoya-u.ac.jp}

\begin{abstract}
    We introduce induced fiber topological pressure for random dynamical systems and identify it with two critical exponents and the threshold of classical fiber pressure, allowing infinite fiber topological entropy. We establish an almost-sure fiberwise formula and a variational principle, and identify induced equilibrium states with classical equilibrium states. If the fiber topological entropy is finite and the entropy map is upper semicontinuous, an ergodic induced equilibrium state exists.
\end{abstract}
\maketitle

\section{Introduction}
Topological pressure relates orbit complexity to invariant measures through the variational principle \cite{ruelle,walters2,walters}. Induced topological pressure extends this framework by measuring orbit lengths through Birkhoff sums of a positive scaling function. Related constructions arise in S. V. Savchenko's study of entropy for special flows over countable Markov shifts \cite{savchenko}. J. Jaerisch, M. Kesseb\"{o}hmer, and S. Lamei \cite{jaerisch} developed induced pressure for countable Markov shifts, and Z. Xing and E. Chen \cite{xing1} extended it to continuous maps on compact metric spaces. In the latter setting, orbit lengths are determined by the condition $S_{n}\psi(x)\leq T<S_{n+1}\psi(x)$, and the variational principle expresses induced pressure as a ratio of entropy plus the integral of the potential to the integral of the scaling function.

For random dynamical systems, T. Bogensch\"{u}tz \cite{bogen1} and Y. Kifer \cite{kifer1} established the theory of fiber topological pressure and its variational principle. Related work includes induced measure-theoretic entropy and the Katok formula \cite{yang}, and induced average pressure for skew-product systems \cite{rahimi1,rahimi4}. Here we study induced pressure defined by fiberwise partition sums, with a scaling function that may depend on both the base and fiber coordinates. This requires control of the random time levels and their partition sums, as well as a comparison between averaged and almost-sure growth rates.

We define induced fiber topological pressure using spanning sets and prove an equivalent formulation using separated sets. The pressure can be identified with two critical exponents and the threshold of classical fiber pressure. A variational principle is established later. These results allow infinite fiber topological entropy.

Let $(X,f)$ be a continuous bundle random dynamical system over an invertible ergodic measure-preserving map $\theta$ of a complete Lebesgue probability space $(\Omega,\F,\prob)$, with compact metric fiber $X$. Write $J=\Omega\times X$, and let $\M_{\prob}(f)$ denote the invariant probability measures of the skew-product $\Theta(\omega,x)=(\theta\omega,f_{\omega}x)$ with marginal $\prob$. Denote the relative entropy by $h_{\mu}(f)$ and the classical fiber pressure by $\pres(f,\cdot)$.
\begin{theorem}\label{main}
    Let $\varphi\in L_{\prob}^{1}(\Omega,C(X))$, and let $\psi\in L_{\prob}^{\infty}(\Omega,C(X))$ with \[\essinf_{\omega\in\Omega}\inf_{x\in X}\psi_{\omega}(x)>0.\] Then
    \[\pres_{\psi}(f,\varphi)=\inf\{\beta\in\R\colon\pres(f,\varphi-\beta\psi)\leq0\}=\beta_{c}=s_{c}=\sup\left\{\frac{h_{\mu}(f)+\int_{J}\varphi\,\dd\mu}{\int_{J}\psi\,\dd\mu}\colon\mu\in\M_{\prob}(f)\right\}.\]
\end{theorem}
Here the $\psi$-induced fiber topological pressure $\pres_{\psi}(f,\varphi)$ is defined in \secref{sec3}. The numbers $\beta_{c}$ and $s_{c}$ are the Poincar\'{e}-type exponent and the Laplace-type exponent defined in \secref{sec4.2}, respectively. If $h_{\tp}(f)<\infty$, their common value is the unique zero of $\beta\mapsto\pres(f,\varphi-\beta\psi)$. If $h_{\tp}(f)=\infty$, the common value is infinite, and this pressure function has no finite zero.

The induced pressure also admits an almost-sure representation in terms of separated and spanning partition sums $P_{\psi,T}(f,\omega,\varphi,\epsilon)$ and $Q_{\psi,T}(f,\omega,\varphi,\epsilon)$ defined in \secref{sec3}.
\begin{theorem}
    Under the assumptions of \thmref{main}, \[\pres_{\psi}(f,\varphi)=\lim_{\epsilon\to0}\limsup_{T\to\infty}\frac{1}{T}\log P_{\psi,T}(f,\omega,\varphi,\epsilon)=\lim_{\epsilon\to0}\limsup_{T\to\infty}\frac{1}{T}\log Q_{\psi,T}(f,\omega,\varphi,\epsilon)\] for $\almost$.
\end{theorem}
When the fiber topological entropy is finite, induced equilibrium states for $(\varphi,\psi)$ coincide with classical equilibrium states for $\varphi-\pres_{\psi}(f,\varphi)\psi$. Upper semicontinuity of the entropy map then yields an ergodic induced equilibrium state. Under additional compactness and continuity assumptions, we also compare fiber induced pressure with the induced pressure of the skew-product map.

The paper is organized as follows. \secref{sec2} recalls the necessary preliminaries. \secref{sec3} introduces induced fiber topological pressure. In \secref{sec4}, we characterize this pressure via two critical exponents alongside the classical pressure threshold, and derive the almost-sure fiberwise pressure formula. \secref{sec5} is devoted to proving the variational principle and analyzing induced equilibrium states. \secref{sec6} compares fiber and skew-product pressure. \secref{sec7} presents a symbolic example with a scaling potential depending on the random environment and the fiber coordinate.

\section{Preliminaries}\label{sec2}
\subsection{Classical induced topological pressure}
We recall the notion of induced topological pressure for a continuous map $f\colon X\to X$ on a compact metric space $(X,d)$ \cite{xing1}. Let $\N:=\{1,2,\ldots\}$ and $\varphi,\psi\in C(X,\R)$ with $\psi>0$. For $T>0$ and $n\in\N$, define the induced time level set \[X_{n}=X_{n}(T):=\{x\in X\colon S_{n}\psi(x)\leq T<S_{n+1}\psi(x)\}.\] Note that when $0<T<\inf\psi$, the set $X_n$ is empty. In this case, we set $X_n=\varnothing$. Then define the induced time set \[S_{T}:=\{n\in\N\colon X_{n}(T)\neq\varnothing\}.\] Since $n\inf\psi\leq T$ for $n\in S_{T}$, the set $S_{T}\subseteq\{1,\ldots,\lfloor T/\inf\psi\rfloor\}$ is finite, where $\lfloor\cdot\rfloor$ denotes the floor function. Let
\[Q_{\psi,T}(f,\varphi,\epsilon):=\inf\biggl\{\sum_{n\in S_T}\sum_{x\in F_n}\e^{S_{n}\varphi(x)}\colon F_{n} \text{ is an $(n,\epsilon)$-spanning set of } X_n,\,n\in S_T\biggr\}.\]
When $S_{T}=\varnothing$, set $Q_{\psi,T}(f,\varphi,\epsilon)=1$.
\begin{definition}[{\cite[Definition 1.1]{xing1}}]
    Let $(X,f)$ be a topological dynamical system, and let $\varphi,\psi\in C(X,\R)$ with $\psi>0$. The $\psi$-\emph{induced topological pressure} of the potential $\varphi$ is defined by
    \[\pres_{\psi}(f,\varphi):=\lim_{\epsilon\to0}\limsup_{T\to\infty}\frac{1}{T}\log Q_{\psi,T}(f,\varphi,\epsilon).\]
\end{definition}
The outer limit exists by monotonicity in $\epsilon$. For $\psi\equiv\ind$, this definition reduces to classical topological pressure. Equivalent definitions using separated sets and open covers are given in \cite[Proposition 2.1]{xing1} and \cite[Section 2.1]{ma}, respectively.

We recall the variational principle for classical induced pressure.
\begin{theorem}[{\cite[Theorem 1.1]{xing1}}]\label{inducevp}
    Let $(X,f)$ be a topological dynamical system, and let $\varphi,\psi\in C(X,\R)$ with $\psi>0$. Then
    \[\pres_{\psi}(f,\varphi)=\sup\left\{\frac{h_{\mu}(f)+\int_{X}\varphi\,\dd\mu}{\int_{X}\psi\,\dd\mu}\colon\mu\in\M(X,f)\right\}.\]
\end{theorem}
Induced pressure is also characterized by the threshold of classical pressure.
\begin{proposition}[{\cite[Corollary 3.2 and Corollary 3.3]{xing1}}]\label{tdscritical}
    Let $(X,f)$ be a topological dynamical system, and let $\varphi,\psi\in C(X,\R)$ with $\psi>0$. Then
    \[\pres_{\psi}(f,\varphi)=\inf\{\beta\in\R\colon\pres(f,\varphi-\beta\psi)\leq0\}=\sup\{\beta\in\R\colon\pres(f,\varphi-\beta\psi)\geq0\}.\] In addition, if $\pres(f,\varphi-\beta\psi)<\infty$ for every $\beta\in\R$, then $\pres(f,\varphi-\pres_{\psi}(f,\varphi)\psi)=0$.
\end{proposition}
\subsection{Fiber topological pressure for random dynamical systems}
We recall the framework of random dynamical systems. Let $(\Omega,\F,\prob)$ be a complete Lebesgue probability space, and let $\theta\colon\Omega\to\Omega$ be an invertible ergodic measure-preserving map. We refer to $(\Omega,\F,\prob,\theta)$ as the \emph{base system}. Let $(X,d)$ be a compact metric space with the Borel $\sigma$-algebra $\B$, and equip $J:=\Omega\times X$ with the product $\sigma$-algebra $\F\otimes\B$.

A \emph{\textup{(}continuous bundle\textup{)} random dynamical system} over the base system $(\Omega,\F,\prob,\theta)$ is a family $f=\{f_\omega\colon X\to X\}_{\omega\in\Omega}$ such that $(\omega,x)\mapsto f_{\omega}(x)$ is jointly measurable and $f_{\omega}$ is continuous for $\almost$. We denote this system by $(X,f)$. Write $f_{\omega}^{n}:=f_{\theta^{n-1}\omega}\circ\cdots f_{\theta\omega}\circ f_{\omega}$ and $f_{\omega}^{0}=\id$. The associated measurable \emph{skew-product map} on $J$ is $\Theta(\omega,x):=(\theta\omega,f_{\omega}x)$. Let $\M_{\prob}(J)$ be the set of Borel probability measures on $(J,\F\otimes\B)$ with marginal $\prob$ on $\Omega$. Each $\mu\in\M_{\prob}(J)$ admits a disintegration $\dd\mu(\omega,x)=\dd\mu_{\omega}(x)\,\dd\prob(\omega)$ for $(\omega,x)\in J$. Denote the $\Theta$-invariant measures in this set by $\M_{\prob}(f)$, and its ergodic measures by $\M_{\prob}^{e}(f)$. Invariance is equivalent to $(f_{\omega})_{*}\mu_{\omega}=\mu_{\theta\omega}$ for $\almost$.

For a jointly measurable function $g\colon J\to\R$ with continuous fiber sections $g_{\omega}:=g(\omega,\cdot)$, set $\onenorm{g}:=\int_{\Omega}\supnorm{g_\omega}\,\probinte=\int_{\Omega}\sup_{x\in X}|g_{\omega}(x)|\,\probinte$. Let $(L_{\prob}^{1}(\Omega,C(X)),\onenorm{\cdot})$ be the resulting Banach space of functions with finite norm. Its subspace $L_{\prob}^{\infty}(\Omega,C(X))$ consists of functions with $\supnorm{g}:=\esssup_{\omega\in\Omega}\supnorm{g_\omega}<\infty$. More general measurable bundles are treated in \cite{kifer1,kifer3}.

Given a $\Theta$-invariant probability measure $\mu\in\M_{\prob}(f)$ and a finite measurable partition $\xi$ of $ J$, the measure-theoretic entropy of $f$ with respect to $\xi$ is defined by
\[h_{\mu}(f,\xi):=\lim_{n\to\infty}\frac{1}{n}\int_{\Omega}H_{\mu_\omega}\left(\bigvee_{i=0}^{n-1}(f_{\omega}^{i})^{-1}\xi_{\theta^{i}\omega}\right)\,\probinte,\] where $H_{\mu_\omega}$ is the entropy of the partition and $\xi_\omega$ is the fiber section of $\xi$ on $X$. The \emph{measure-theoretic entropy} of $f$ is $h_{\mu}(f):=\sup_{\xi}h_{\mu}(f,\xi)$, where $\xi$ ranges over all finite measurable partitions of $J$. We use the following affinity property.
\begin{proposition}[{\cite[Lemma 5.3]{wang}}]\label{ranentropy}
    The entropy map $\mu\mapsto h_{\mu}(f)$ is affine.
\end{proposition}
For each $n\geq1$, $\omega\in\Omega$, and $x,y\in X$, the associated Bowen metric is \[d_{n}^{\omega}(x,y):=\max\{d(f_{\omega}^{i}(x),f_{\omega}^{i}(y))\colon i=0,1,\ldots,n-1\}.\] The map $(\omega,x,y)\mapsto d_{n}^{\omega}(x,y)$ is jointly measurable. Write $B_{n}^{\omega}(x,\epsilon):=\{y\in X\colon d_{n}^{\omega}(x,y)<\epsilon\}$. A subset $E\subset X$ is $(\omega,n,\epsilon)$-separated if $d_{n}^{\omega}(x,y)>\epsilon$ for any two distinct points $x,y\in E$. A subset $F\subset X$ is $(\omega,n,\epsilon)$-spanning if for each $x\in X$ there exists $y\in F$ such that $d_{n}^{\omega}(x,y)\leq\epsilon$. Let $\varphi\in L_{\prob}^{1}(\Omega,C(X))$. Set
\begin{equation}\label{spanning}
    Q_{n}(f,\omega,\varphi,\epsilon):=\inf\left\{\sum_{x\in F_n}\e^{S_{n}\varphi_{\omega}(x)}\colon F_n\text{ is an $(\omega,n,\epsilon)$-spanning set of }X\right\}
\end{equation} and
\begin{equation}\label{separated}
    P_{n}(f,\omega,\varphi,\epsilon):=\sup\left\{\sum_{x\in E_n}\e^{S_{n}\varphi_{\omega}(x)}\colon E_n\text{ is an $(\omega,n,\epsilon)$-separated set of }X\right\}.
\end{equation}
Here $S_{n}\varphi_{\omega}:=\sum_{i=0}^{n-1}\varphi_{\theta^{i}\omega}\circ f_{\omega}^{i}$. For fixed $n\geq1$ and $\epsilon>0$, both partition sums are measurable functions of $\omega$ \cite[Lemma 5.3]{bogen1}.
\begin{definition}\label{ranpres}
    Let $(X,f)$ be a random dynamical system. Let $\varphi\in L_{\prob}^{1}(\Omega,C(X))$. The \emph{fiber topological pressure} of the potential $\varphi$ is defined by
    \[\pres(f,\varphi)=\lim_{\epsilon\to0}\limsup_{n\to\infty}\frac{1}{n}\int_{\Omega}\log Q_{n}(f,\omega,\varphi,\epsilon)\,\probinte=\lim_{\epsilon\to0}\limsup_{n\to\infty}\frac{1}{n}\int_{\Omega}\log P_{n}(f,\omega,\varphi,\epsilon)\,\probinte.\]
\end{definition}
\begin{remark}
    When $\varphi\equiv0$, the quantity $h_{\tp}(f):=\pres(f,0)$ is called the \emph{fiber topological entropy} of $f$, or the relative topological entropy of $\Theta$.
\end{remark}
Ergodicity of the base yields the following almost-sure representation; see \cite[Theorem 3.2]{kifer2}.
\begin{theorem}[{\cite[Proposition 1.6]{kifer1}}]\label{randompres}
    \[\pres(f,\varphi)=\lim_{\epsilon\to0}\limsup_{n\to\infty}\frac{1}{n}\log Q_{n}(f,\omega,\varphi,\epsilon)=\lim_{\epsilon\to0}\limsup_{n\to\infty}\frac{1}{n}\log P_{n}(f,\omega,\varphi,\epsilon)\] for $\almost$.
\end{theorem}
The fiber pressure satisfies the variational principle established by T. Bogensch\"{u}tz \cite{bogen1} and Y. Kifer \cite{kifer1}.
\begin{theorem}[{\cite[Theorem 6.1]{bogen1}}]\label{ranvp}
    \[\pres(f,\varphi)=\sup\left\{h_{\mu}(f)+\int_{ J}\varphi\,\dd\mu\colon\mu\in\M_{\prob}(f)\right\}.\]
\end{theorem}
For the zero potential, this gives $h_{\tp}(f)=\sup\{h_{\mu}(f)\colon\mu\in\M_{\prob}(f)\}$. A measure attaining the supremum in \thmref{ranvp} is an equilibrium state for $\varphi$; denote the set of all equilibrium states for $\varphi$ by $\E_{\prob}(\varphi,f)$.
\section{Induced pressure for random dynamical systems}\label{sec3}
Throughout the paper, let $\varphi\in L_{\prob}^{1}(\Omega,C(X))$ and let$\psi\in L_{\prob}^{\infty}(\Omega,C(X))$ with \[m:=\essinf_{\omega\in\Omega}\inf_{x\in X}\psi_{\omega}(x)>0.\] Choose a measurable set $\Omega_{0}$ of full measure on which the fiber maps and potentials are continuous and $m\leq\psi_{\omega}\leq\supnorm{\psi}$. Then $\Omega_{*}:=\bigcap_{k\in\Z}\theta^{-k}\Omega_{0}$ is invariant and has full measure. We define the fiberwise partition sums on $\Omega_{*}$ and extend them measurably to its complement.

For each $\omega\in\Omega_{*}$ and $T>0$, define \[X_{\omega,n}=X_{\omega,n}(T):=\{x\in X\colon S_{n}\psi_{\omega}(x)\leq T< S_{n+1}\psi_{\omega}(x)\}\] for each $n\in\N$. When $0<T<m$ or such points do not exist, we set $X_{\omega,n}=\varnothing$. Then define the induced time set \[S_{\omega,T}:=\{n\in\N\colon X_{\omega,n}(T)\neq\varnothing\}.\] Since $nm\leq T$ whenever $n\in S_{\omega,T}$, we have $S_{\omega,T}\subseteq\{1,\ldots,N_{T}\}$, where $N_{T}:=\lfloor T/m\rfloor$.

For $j\geq 1$, define \[X_{\omega,n}^{(j)}:=\left\{x\in X\colon S_{n}\psi_{\omega}(x)\leq T, S_{n+1}\psi_{\omega}(x)\geq T+\frac{1}{j}\right\}.\] Each $X_{\omega,n}^{(j)}$ is compact and $X_{\omega,n}=\bigcup_{j=1}^{\infty}X_{\omega,n}^{(j)}$. Let $\F_{*}:=\{B\cap\Omega_{*}\colon B\in\F\}$. For every open set $U\subset X$, \[\{\omega\in\Omega_{*}\colon X_{\omega,n}\cap U\neq\varnothing\}=\bigcup_{j=1}^{\infty}\left\{\omega\in\Omega_{*}\colon X_{\omega,n}^{(j)}\cap U\neq\varnothing\right\}.\] Define $r_{j}(\omega,x):=\max\{S_{n}\psi_{\omega}(x)-T,T+1/j-S_{n+1}\psi_{\omega}(x)\}$. Then $X_{\omega,n}^{(j)}=\{x\in X\colon r_{j}(\omega,x)\leq0\}$. Let $\{F_{k}\}_{k\geq1}$ be compact subsets of $U$ such that $U=\bigcup_{k\geq1}F_{k}$. For example, one may take $F_{k}=\{x\in X\colon d(x,X\setminus U)\geq 1/k\}$ for each $k\geq1$ and $F_{k}=X$ when $U=X$. Then $X_{\omega,n}^{(j)}\cap U\neq\varnothing$ is equivalent to $\min_{x\in F_k}r_{j}(\omega,x)\leq0$ for some $k\geq1$. Since $r_{j}(\omega,\cdot)$ is continuous, the map $\omega\mapsto\min_{x\in F_k}r_{j}(\omega,x)$ is measurable. Consequently, $\{\omega\in\Omega_{*}\colon X_{\omega,n}^{(j)}\cap U\neq\varnothing\}\in\F_{*}$ for every open $U\subset X$. Thus the multifunction $\omega\mapsto X_{\omega,n}^{(j)}$ is measurable, and so is the increasing union $\omega\mapsto X_{\omega,n}$. Since $\{\omega\in\Omega_{*}\colon n\in S_{\omega,T}\}=\{\omega\in\Omega_{*}\colon X_{\omega,n}\neq\varnothing\}$, the finite-valued multifunction $\omega\mapsto S_{\omega,T}$ is measurable. The time levels are pairwise disjoint and cover $X$ whenever $T\geq\supnorm{\psi}$. For $\omega\in\Omega_{*}$, let 
\begin{multline}\label{def_Q}
    Q_{\psi,T}(f,\omega,\varphi,\epsilon):=\inf\biggl\{\sum_{n\in S_{\omega,T}}\sum_{x\in F_n}\e^{S_{n}\varphi_{\omega}(x)}\colon \\
    F_{n}\text{ is an $(\omega,n,\epsilon)$-spanning set of }X_{\omega,n},\,n\in S_{\omega,T}\biggr\},
\end{multline}
and let 
\begin{multline}\label{def_P}
    P_{\psi,T}(f,\omega,\varphi,\epsilon):=\sup\biggl\{\sum_{n\in S_{\omega,T}}\sum_{x\in E_n}\e^{S_{n}\varphi_{\omega}(x)}\colon \\
    E_{n}\text{ is an $(\omega,n,\epsilon)$-separated set of }X_{\omega,n},\,n\in S_{\omega,T}\biggr\}.
\end{multline}
In particular, set $Q_{\psi,T}(f,\omega,\varphi,\epsilon)=P_{\psi,T}(f,\omega,\varphi,\epsilon)=1$ when $S_{\omega,T}=\varnothing$. For $\omega\in\Omega\setminus\Omega_{*}$, we simply define $Q_{\psi,T}(f,\omega,\varphi,\epsilon)=P_{\psi,T}(f,\omega,\varphi,\epsilon)=1$. Fix $T>0$ and $\epsilon>0$. For $n\leq N_{T}$, let $q_{n,T}(f,\omega,\varphi,\epsilon)$ and $p_{n,T}(f,\omega,\varphi,\epsilon)$ be the spanning and separated partition sums on $X_{\omega,n}(T)$ in \eqref{spanning} and \eqref{separated}, respectively, with both set to zero when the time level is empty.
\begin{lemma}\label{meas}
    The maps $\omega\mapsto Q_{\psi,T}(f,\omega,\varphi,\epsilon)$ and $\omega\mapsto P_{\psi,T}(f,\omega,\varphi,\epsilon)$ are measurable.
\end{lemma}
\begin{proof}
    The argument relies on measurable multifunctions and the measurable projection theorem from \cite{castaing}. Fix $\omega\in\Omega_{*}$. Recall from the preceding construction that $X_{\omega,n}=\bigcup_{j\geq1}X_{\omega,n}^{(j)}$, where $X_{\omega,n}^{(j)}$ are increasing compact-valued measurable multifunctions possibly taking the empty value. For $k,j\geq1$, define $A_{j}:=\{\omega\in\Omega_{*}\colon X_{\omega,n}^{(j)}\neq\varnothing\}$, and define \[H_{n,k}^{(j)}(\omega,\mathbf{x}):=\max_{y\in X_{\omega,n}^{(j)}}\min_{1\leq i\leq k}d_{n}^{\omega}(x_i,y)\] for $\omega\in A_j$ and $H_{n,k}^{(j)}(\omega,\mathbf{x})=0$ for $\omega\in\Omega_{*}\setminus A_j$, where $\mathbf{x}=(x_1,\ldots,x_k)\in X^{k}$. Since the multifunction $\omega\mapsto X_{\omega,n}^{(j)}$ is measurable, the set $A_{j}=\{\omega\in\Omega_{*}\colon X_{\omega,n}^{(j)}\cap X\neq\varnothing\}$ belongs to $\F_{*}$. On the measurable set $A_j$, the Castaing representation theorem \cite[Theorem III.9]{castaing} provides measurable selections $s_{\ell}^{(j)}\colon A_{j}\to X$ whose values are dense in $X_{\omega,n}^{(j)}$. Continuity in the fiber variable implies that $H_{n,k}^{(j)}(\omega,\mathbf{x})$ is jointly measurable. Extending it by $0$ outside $A_j$, and taking the supremum over $j$, yields the measurability of $H_{n,k}(\omega,\mathbf{x}):=\sup_{y\in X_{\omega,n}}\min_{1\leq i\leq k}d_{n}^{\omega}(x_i,y)=\sup_{j\geq1}H_{n,k}^{(j)}(\omega,\mathbf{x})$ with the convention $H_{n,k}=0$ when $X_{\omega,n}=\varnothing$. Define \[E_{k}^{n,\epsilon}:=\{(\omega,\mathbf{x})\in\Omega_{*}\times X^{k}\colon x_{i}\in X_{\omega,n}, 1\leq i\leq k, H_{n,k}(\omega,\mathbf{x})\leq\epsilon\},\] and \[F_{k}^{n,a}:=\left\{(\omega,\mathbf{x})\in\Omega_{*}\times X^{k}\colon\sum_{i=1}^{k}\e^{S_{n}\varphi_{\omega}(x_i)}<a\right\}\] for $a>0$. The graphs of $X_{\omega,n}^{(j)}$ and $X_{\omega,n}$ are measurable, since they are defined by inequalities involving jointly measurable Birkhoff sums. By the measurability of the Bowen metric and the Birkhoff sums, $E_{k}^{n,\epsilon}\cap F_{k}^{n,a}\in\F_{*}\otimes\B^{k}$. Moreover, \[\{\omega\in\Omega_{*}\colon X_{\omega,n}=\varnothing\}\cup\bigcup_{k\geq1}\pi_{\Omega}(E_{k}^{n,\epsilon}\cap F_{k}^{n,a})=\{\omega\in\Omega_{*}\colon q_{n,T}(f,\omega,\varphi,\epsilon)<a\}.\] Since $(\Omega_{*},\F_{*},\prob|_{\Omega_{*}})$ is complete and $X^{k}$ is Polish, the measurable projection theorem \cite[Theorem III.23]{castaing} implies that $q_{n,T}(f,\omega,\varphi,\epsilon)$ is $\F_{*}$-measurable.

    Next, consider the separated partition sums. For $k\geq1$ and $a>0$, define \[U_{k}^{n,\epsilon}:=\left\{(\omega,\mathbf{x})\in\Omega_{*}\times X^{k}\colon x_{i}\in X_{\omega,n},1\leq i\leq k,d_{n}^{\omega}(x_i,x_j)>\epsilon, i\neq j \right\}\] and \[V_{k}^{n,a}:=\left\{(\omega,\mathbf{x})\in\Omega_{*}\times X^{k}\colon\sum_{i=1}^{k}\e^{S_{n}\varphi_{\omega}(x_i)}>a\right\}.\] Both sets belong to $\F_{*}\otimes\B^{k}$, and \[\bigcup_{k\geq1}\pi_{\Omega}(U_{k}^{n,\epsilon}\cap V_{k}^{n,a})=\{\omega\in\Omega_{*}\colon p_{n,T}(f,\omega,\varphi,\epsilon)>a\}.\] The same theorem shows that $p_{n,T}(f,\omega,\varphi,\epsilon)$ is $\F_{*}$-measurable.

    Summing over $n\in S_{\omega,T}$, we have $Q_{\psi,T}(f,\omega,\varphi,\epsilon)=\sum_{n=1}^{N_T}q_{n,T}(f,\omega,\varphi,\epsilon)+\ind_{\{S_{\omega,T}=\varnothing\}}$ and $P_{\psi,T}(f,\omega,\varphi,\epsilon)=\sum_{n=1}^{N_T}p_{n,T}(f,\omega,\varphi,\epsilon)+\ind_{\{S_{\omega,T}=\varnothing\}}$. Both partition sums are measurable on $\Omega_{*}$. Their prescribed extensions to $\Omega\setminus\Omega_{*}$ are therefore $\F$-measurable.
\end{proof}
\begin{definition}\label{randominduced}
    Let $(X,f)$ be a random dynamical system, and let $\varphi\in L_{\prob}^{1}(\Omega,C(X))$, $\psi\in L_{\prob}^{\infty}(\Omega,C(X))$ with $\essinf_{\omega\in\Omega}\inf_{x\in X}\psi_{\omega}(x)>0$. The \emph{$\psi$-induced fiber topological pressure} of the potential $\varphi$ is given by
    \begin{equation}\label{inducedlimit}
        \pres_{\psi}(f,\varphi):=\lim_{\epsilon\to0}\limsup_{T\to\infty}\frac{1}{T}\int_{\Omega}\log Q_{\psi,T}(f,\omega,\varphi,\epsilon)\,\probinte.
    \end{equation}
\end{definition}
For $\psi\equiv\ind$, this definition reduces to classical fiber pressure in \defref{ranpres}. We now verify that its defining integrals are finite for every sufficiently large $T$ and fixed $\epsilon>0$. Fix a sufficiently large $T>0$ and $\omega\in\Omega_{*}$. Since $X$ is compact, there exists a finite $(\omega,n,\epsilon)$-spanning set $F_n$ of $X_{\omega,n}$ such that $\# F_n\geq1$ for every $n\in S_{\omega,T}$. Then 
\begin{equation}\label{lowerbound_Q}
    \sum_{n\in S_{\omega,T}}\sum_{x\in F_n}\e^{S_{n}\varphi_{\omega}(x)}\geq\sum_{n\in S_{\omega,T}}\exp\left(-\sum_{i=0}^{n-1}\supnorm{\varphi_{\theta^{i}\omega}}\right)\geq\exp\left(-\sum_{i=0}^{N_{T}-1}\supnorm{\varphi_{\theta^{i}\omega}}\right).
\end{equation} Taking the infimum over all admissible spanning sets and then taking logarithms gives \[\log Q_{\psi,T}(f,\omega,\varphi,\epsilon)\geq-\sum_{i=0}^{N_{T}-1}\supnorm{\varphi_{\theta^{i}\omega}}.\] Integrating both sides over $\Omega$, we obtain
\[\int_{\Omega}\log Q_{\psi,T}(f,\omega,\varphi,\epsilon)\,\probinte\geq-\sum_{i=0}^{N_{T}-1}\int_{\Omega}\supnorm{\varphi_{\theta^{i}\omega}}\,\probinte=-N_{T}\onenorm{\varphi}>-\infty\]
by the invariance of $\prob$. Furthermore, if $0<\epsilon_{1}<\epsilon_{2}$, then $Q_{\psi,T}(f,\omega,\varphi,\epsilon_1)\geq Q_{\psi,T}(f,\omega,\varphi,\epsilon_2)$. Thus, the limit exists and is bounded below by a finite constant.

It remains to show that $\log Q_{\psi,T}(f,\omega,\varphi,\epsilon)$ is integrable. We first show a useful estimate. For $\omega\in\Omega_{*}$ and $n\in S_{\omega,T}$, choose a maximal $(\omega,n,\epsilon)$-separated set $E_{n}$ of $X_{\omega,n}$ with respect to inclusion. Then it is an $(\omega,n,\epsilon)$-spanning set of $X_{\omega,n}$, leading to $Q_{\psi,T}(f,\omega,\varphi,\epsilon)\leq P_{\psi,T}(f,\omega,\varphi,\epsilon)$. Since $X$ is compact, there is a finite measurable partition $\{Q_{i}\}_{i=1}^{N_{\epsilon}}$ with $\diam(Q_i)<\epsilon$. For every $\omega\in\Omega_{*}$ and $n\geq1$, the cardinality of any $(\omega,n,\epsilon)$-separated set is at most $N_{\epsilon}^{n}$. Recall that $n\leq N_{T}$ for all $n\in S_{\omega,T}$. An argument similar to that used in \eqref{lowerbound_Q} gives
\begin{equation}\label{upperbound}
    P_{\psi,T}(f,\omega,\varphi,\epsilon)\leq N_{T}N_{\epsilon}^{N_{T}}\exp(\sum_{i=0}^{N_{T}-1}\supnorm{\varphi_{\theta^{i}\omega}}).
\end{equation}
As a consequence, \[\int_{\Omega}\log P_{\psi,T}(f,\omega,\varphi,\epsilon)\,\probinte\leq\log N_{T}+N_{T}\log N_{\epsilon}+N_{T}\onenorm{\varphi}<\infty.\] Thus $\log Q_{\psi,T}(f,\omega,\varphi,\epsilon)$ is integrable for every sufficiently large $T$ and every fixed $\epsilon>0$. Moreover, the partition sum is nonincreasing in $\epsilon$, so the outer limit in \eqref{inducedlimit} exists in the extended real line. The lower estimate rules out the value $-\infty$. Thus the induced fiber topological pressure is well-defined.

The next proposition gives an equivalent formulation in terms of separated sets.
\begin{proposition}\label{equivalent}
    We have \[\pres_{\psi}(f,\varphi)=\lim_{\epsilon\to0}\limsup_{T\to\infty}\frac{1}{T}\int_{\Omega}\log P_{\psi,T}(f,\omega,\varphi,\epsilon)\,\probinte.\]
\end{proposition}
\begin{proof}
    Since $Q_{\psi,T}(f,\omega,\varphi,\epsilon)\leq P_{\psi,T}(f,\omega,\varphi,\epsilon)$, only the reverse pressure inequality remains to be proved. For $\epsilon>0$, let \[v_{\omega}(\varphi,\epsilon):=\sup\{|\varphi_{\omega}(x)-\varphi_{\omega}(y)|\colon x,y\in X, d(x,y)\leq \epsilon\},\] and let $v(\varphi,\epsilon):=\int_{\Omega}v_{\omega}(\varphi,\epsilon)\,\probinte$. The function $\varphi_{\omega}$ is uniformly continuous on $X$ for $\almost$; hence $\lim_{\epsilon\to0}v_{\omega}(\varphi,\epsilon)=0$. The map $\omega\mapsto v_{\omega}(\varphi,\epsilon)$ is measurable, since it can be written as $\sup_{j\geq1}|\varphi_{\omega}(x_j)-\varphi_{\omega}(y_j)|$, where $\{(x_{j},y_{j})\}_{j\geq1}$ is a countable dense subset of the compact set $\{(x,y)\in X^{2}\colon d(x,y)\leq\epsilon\}$. Since $0\leq v_{\omega}(\varphi,\epsilon)\leq2\supnorm{\varphi_{\omega}}$, $v_{\omega}(\varphi,\epsilon)\in L^{1}(\prob)$ and, consequently, $\lim_{\epsilon\to0}v(\varphi,\epsilon)=0$ by the dominated convergence theorem.

    For a sufficiently large $T>0$, fix $\omega\in\Omega_{*}$ and $\epsilon>0$. For $n\in S_{\omega,T}$, choose any $(\omega,n,2\epsilon)$-separated set $E_{n}$ of $X_{\omega,n}(T)$ and any $(\omega,n,\epsilon)$-spanning set $F_{n}$ of $X_{\omega,n}(T)$. For each $x\in E_{n}$, choose $i_{n}(x)\in F_{n}$ with $d_{n}^{\omega}(x,i_{n}(x))\leq\epsilon$. Because $E_{n}$ is $2\epsilon$-separated and $F_{n}$ is $\epsilon$-spanning, the assignment $i_{n}\colon E_{n}\to F_{n}$ is injective. For all $0\leq i\leq n-1$, we have $d(f_{\omega}^{i}(x),f_{\omega}^{i}(i_{n}(x)))\leq\epsilon$, so $|\varphi_{\theta^{i}\omega}(f_{\omega}^{i}(x))-\varphi_{\theta^{i}\omega}(f_{\omega}^{i}(i_{n}(x)))|\leq v_{\theta^{i}\omega}(\varphi,\epsilon)$. Hence $|S_{n}\varphi_{\omega}(x)-S_{n}\varphi_{\omega}(i_{n}(x))|\leq\sum_{i=0}^{n-1}v_{\theta^{i}\omega}(\varphi,\epsilon)\leq\sum_{i=0}^{N_{T}-1}v_{\theta^{i}\omega}(\varphi,\epsilon)$. Define
    \begin{equation}\label{modulus}
        V_{T}(\omega,\epsilon):=\sum_{i=0}^{N_{T}-1}v_{\theta^{i}\omega}(\varphi,\epsilon).
    \end{equation} Then $S_{n}\varphi_{\omega}(x)\leq S_{n}\varphi_{\omega}(i_{n}(x))+V_{T}(\omega,\epsilon)$. Summing over $x$ and $n$, then taking the supremum over separated sets and the infimum over spanning sets gives 
    \begin{equation}\label{comparison}
        P_{\psi,T}(f,\omega,\varphi,2\epsilon)\leq \e^{V_{T}(\omega,\epsilon)}Q_{\psi,T}(f,\omega,\varphi,\epsilon).
    \end{equation}
    Taking logarithms and integrating yields
    \begin{equation}\label{integralcomparison}
        \begin{aligned}[b]
            \frac{1}{T}\int_{\Omega}\log P_{\psi,T}(f,\omega,\varphi,2\epsilon)\,\probinte&\leq\frac{1}{T}\int_{\Omega}\log Q_{\psi,T}(f,\omega,\varphi,\epsilon)\,\probinte+\frac{1}{T}\int_{\Omega}V_{T}(\omega,\epsilon)\,\probinte \\
            &\leq\frac{1}{T}\int_{\Omega}\log Q_{\psi,T}(f,\omega,\varphi,\epsilon)\,\probinte+\frac{1}{m}v(\varphi,\epsilon),
        \end{aligned}
    \end{equation} where the last inequality comes from the invariance of $\prob$ and the definition of $N_{T}$. The error term $m^{-1}v(\varphi,\epsilon)$ tends to zero as $\epsilon\to0$. Letting $T\to\infty$ and $\epsilon\to0$ gives \[\lim_{\epsilon\to0}\limsup_{T\to\infty}\frac{1}{T}\int_{\Omega}\log P_{\psi,T}(f,\omega,\varphi,2\epsilon)\,\probinte\leq\pres_{\psi}(f,\varphi).\] The proof is complete.
\end{proof}
\section{Critical exponents and pressure characterization}\label{sec4}
We retain the standing assumptions and allow infinite fiber topological entropy unless stated otherwise.
\subsection{Unique zero}
We begin by examining the zeros of the classical fiber pressure function $\pres(\beta):=\pres(f,\varphi-\beta\psi)$.
\begin{proposition}\label{preszero}
    (1) If $h_{\tp}(f)<\infty$, then $\pres(\beta)\colon\R\to\R$ is continuous and strictly decreasing, and has a unique zero $\beta_{0}\in\R$.

    (2) If $h_{\tp}(f)=\infty$, then $\pres(\beta)=\infty$ for all $\beta\in\R$ and $\beta_{0}=\infty$.
\end{proposition}
\begin{proof}
    Suppose $h_{\tp}(f)<\infty$. The classical variational principle gives \[-\onenorm{\varphi}-|\beta|\supnorm{\psi}\leq\pres(\beta)\leq h_{\tp}(f)+\onenorm{\varphi}+|\beta|\supnorm{\psi}.\] Thus $\pres(\beta)$ is finite for every $\beta\in\R$. By the variational principle in \thmref{ranvp}, we have
    \[\pres(\beta)=\sup\left\{h_{\mu}(f)+\int_{J}(\varphi-\beta\psi)\,\dd\mu\colon\mu\in\M_{\prob}(f)\right\}.\] For $\beta_{1}<\beta_{2}$, the same principle yields \[\pres(\beta_1)-\supnorm{\psi}(\beta_2-\beta_1)\leq\pres(\beta_2)\leq\pres(\beta_1)-m(\beta_2-\beta_1).\] Then the map $\beta\mapsto\pres(\beta)$ is strictly decreasing and continuous in $\beta$. Since $\lim_{\beta\to\infty}\pres(\beta)=-\infty$ and $\lim_{\beta\to-\infty}\pres(\beta)=\infty$, it has a unique zero $\beta_{0}$. It follows that \[\beta_{0}=\inf\{\beta\in\R\colon\pres(f,\varphi-\beta\psi)\leq0\}=\sup\{\beta\in\R\colon\pres(f,\varphi-\beta\psi)\geq0\}.\]
    If $h_{\tp}(f)=\infty$, the variational principle gives $\pres(\beta)=\infty$ for every $\beta\in\R$, and we set $\beta_{0}=\inf\varnothing=\infty$.
\end{proof}
\begin{lemma}\label{presentropy}
    The induced fiber pressure satisfies \[\pres_{\psi}(f,\varphi)\geq\frac{1}{\supnorm{\psi}}h_{\tp}(f)-\frac{\onenorm{\varphi}}{m}.\] In particular, $\pres_{\psi}(f,\varphi)=\infty$ whenever $h_{\tp}(f)=\infty$.
\end{lemma}
\begin{proof}
    Choose an $(\omega,n,\epsilon)$-separated set $E_{n}\subseteq X$ of maximal cardinality $s_{n}(\omega,\epsilon)<\infty$, and set $T_{n}:=n\supnorm{\psi}$. Every point of $E_{n}$ has induced time $k_{T_n}:=\max\{n\in\N\colon S_{n}\psi_{\omega}(x)\leq T_n\}$ at least $n$. Partitioning $E_{n}$ by induced time therefore gives admissible separated sets $\{E_{n,k}\}_{k\geq n}$ for $P_{\psi,T_{n}}(f,\omega,\varphi,\epsilon)$. We have \[P_{\psi,T_n}(f,\omega,\varphi,\epsilon)\geq\sum_{k\geq n}\sum_{x\in E_{n,k}}\e^{S_{k}\varphi_{\omega}(x)}=\sum_{x\in E_n}\e^{S_{k_{T_n}}\varphi_{\omega}(x)}.\] Recall that $S_{k_{T_n}}\varphi_{\omega}(x)\geq-\sum_{i=0}^{k_{T_n}-1}\supnorm{\varphi_{\theta^{i}\omega}}\geq-\sum_{i=0}^{\lfloor n\supnorm{\psi}/m\rfloor-1}\supnorm{\varphi_{\theta^{i}\omega}}$. Thus \[P_{\psi,T_n}(f,\omega,\varphi,\epsilon)\geq s_{n}(\omega,\epsilon)\exp(-\sum_{i=0}^{\lfloor n\supnorm{\psi}/m\rfloor-1}\supnorm{\varphi_{\theta^{i}\omega}}).\] Taking logarithms, integrating, and passing to the upper limit along $T_{n}$, using invariance of $\prob$ yields
    \begin{equation}\label{entropylowerbound}
        \limsup_{T\to\infty}\frac{1}{T}\int_{\Omega}\log P_{\psi,T}(f,\omega,\varphi,\epsilon)\,\probinte\geq\frac{1}{\supnorm{\psi}}\limsup_{n\to\infty}\frac{1}{n}\int_{\Omega}\log s_{n}(\omega,\epsilon)\,\probinte-\frac{\onenorm{\varphi}}{m}.
    \end{equation}
    Recall the definition of fiber topological entropy via separated sets: \[h_{\tp}(f)=\lim_{\epsilon\to0}\limsup_{n\to\infty}\frac{1}{n}\int_{\Omega}\log s_{n}(\omega,\epsilon)\,\probinte.\] Letting $\epsilon\to0$ proves the bound, including the infinite-entropy case. 
\end{proof}
Motivated by \cite[Theorem 2.1]{jaerisch} and \cite[Theorem 3.1]{xing1}, we next identify the induced fiber pressure with the unique zero of the classical fiber pressure function when $h_{\tp}(f)<\infty$, treating the upper and lower bounds separately.
\begin{proposition}\label{ineq1}
    $\pres_{\psi}(f,\varphi)\leq\beta_{0}$.
\end{proposition}
\begin{proof}
    Fix $\beta>\beta_{0}$ and write $\pres(\beta)=p<0$. Define the finite time set \[I_{T}:=\left\{n\in\N\colon\frac{T}{\supnorm{\psi}}-1<n\leq\frac{T}{m}\right\}.\] If $x\in X_{\omega,n}(T)$, then $n\in I_{T}$ and $|S_{n}\psi_{\omega}(x)-T|\leq\supnorm{\psi}$. Comparing the weights and summing over induced times gives
    \begin{equation}\label{expupperbound1}
        P_{\psi,T}(f,\omega,\varphi,\epsilon)\leq \e^{\beta T+|\beta|\supnorm{\psi}}\sum_{n\in I_{T}}P_{n}(f,\omega,\varphi-\beta\psi,\epsilon).
    \end{equation}
    Define $A_{T}(\omega,\epsilon):=\sum_{n\in I_{T}}P_{n}(f,\omega,\varphi-\beta\psi,\epsilon)$. Then 
    \begin{equation}\label{upperboundfrac}
        \frac{1}{T}\int_{\Omega}\log P_{\psi,T}(f,\omega,\varphi,\epsilon)\,\probinte\leq\beta+\frac{|\beta|\supnorm{\psi}}{T}+\frac{1}{T}\int_{\Omega}\log A_{T}(\omega,\epsilon)\,\probinte.
    \end{equation}
    \thmref{randompres} implies that, for each fixed $\epsilon>0$ and $\almost$, \[\limsup_{n\to\infty}\frac{1}{n}\log P_{n}(f,\omega,\varphi-\beta\psi,\epsilon)\leq p.\] Hence there exists $N_{\epsilon}(\omega)\geq1$ such that $P_{n}(f,\omega,\varphi-\beta\psi,\epsilon)\leq \e^{pn/2}$ for all $n\geq N_{\epsilon}(\omega)$. Since $\min I_{T}\to\infty$ for sufficiently large $T$, the geometric series bound 
    \begin{equation}\label{expupperbound2}
        A_{T}(\omega,\epsilon)\leq\sum_{n\in I_{T}}\e^{\frac{p}{2}n}\leq\sum_{n>T/\supnorm{\psi}-1}\e^{\frac{p}{2}n}\leq\frac{\e^{-p/2}}{1-\e^{p/2}}\e^{\frac{pT}{2\supnorm{\psi}}}\leq C_{p}\e^{\frac{pT}{2\supnorm{\psi}}}
    \end{equation}
    follows for some constant $C_{p}>0$ independent of $T$. Thus \[\limsup_{T\to\infty}\frac{1}{T}\log A_{T}(\omega,\epsilon)\leq\frac{p}{2\supnorm{\psi}}<0.\]
    
    It remains to control the positive part of $B_{T}(\omega):=T^{-1}\log A_{T}(\omega,\epsilon)$. Set $h(\omega):=\supnorm{\varphi_{\omega}-\beta\psi_{\omega}}$. The ergodic averages of $h$ converge almost surely and in $L^1$, the latter by Scheff\'{e}'s lemma. They are therefore uniformly integrable. The cardinality estimate used in \eqref{upperbound} gives \[P_{n}(f,\omega,\varphi-\beta\psi,\epsilon)\leq N_{\epsilon}^{n}\exp(\sum_{i=0}^{n-1}h(\theta^{i}\omega)).\] Take $L_{T}:=\max\{1,N_{T}\}$. Since $n\leq L_{T}$ when $n\in I_{T}$, we have \[A_{T}(\omega,\epsilon)\leq\# I_{T}N_{\epsilon}^{L_T}\exp(\sum_{i=0}^{L_{T}-1}h(\theta^{i}\omega)).\] Hence \[B_{T}^{+}(\omega)\leq\frac{\log\# I_{T}}{T}+\frac{L_{T}}{T}\log N_{\epsilon}+\frac{1}{T}\sum_{i=0}^{L_{T}-1}h(\theta^{i}\omega).\] Thus $\{B_{T}^{+}\}_{T\geq T_{0}}$ is uniformly integrable for some $T_{0}>0$. Since $B_{T}^{+}\to0$ almost surely, applying Vitali's convergence theorem along an arbitrary sequence $\{T_{j}\}_{j\geq1}$ tending to infinity gives $\int_{\Omega}B_{T}^{+}\,\dd\prob\to0$. The desired upper bound follows from \eqref{upperboundfrac}, first letting $T\to\infty$ and $\epsilon\to0$, then letting $\beta\to\beta_{0}$.
\end{proof}
\begin{proposition}\label{ineq2}
    $\pres_{\psi}(f,\varphi)\geq\beta_{0}$.
\end{proposition}
\begin{proof}
    By the ergodic form of the variational principle (see the discussion following \cite[Theorem 3.1]{liu}), we can restrict the variational principle to ergodic invariant measures, that is, \[\pres(\beta)=\sup\left\{h_{\mu}(f)+\int_{J}\varphi\,\dd\mu-\beta\int_{J}\psi\,\dd\mu\colon\mu\in\M_{\prob}^{e}(f)\right\}.\] Since $\pres(\beta_{0})=0$ and $\int_{J}\psi\,\dd\mu\geq m$, it follows that \[\beta_{0}=\sup_{\mu\in\M_{\prob}^{e}(f)}\frac{h_{\mu}(f)+\int_{J}\varphi\,\dd\mu}{\int_{J}\psi\,\dd\mu}.\] Indeed, every ratio is at most $\beta_{0}$, while a measure whose normalized free energy $h_{\mu}(f)+\int_{J}(\varphi-\beta_{0}\psi)\,\dd\mu$ exceeds $-\eta$ has ratio greater than $\beta_{0}-\eta/m$ for every $\eta>0$. It therefore suffices to prove the claimed lower bound for an arbitrary ergodic measure $\mu\in\M_{\prob}^{e}(f)$. Fix $\mu\in\M_{\prob}^{e}(f)$, and write $\mu(\psi)=\int_{J}\psi\,\dd\mu$, $\mu(\varphi)=\int_{J}\varphi\,\dd\mu$, and $h=h_{\mu}(f)$. Set $c_{\mu}:=(h+\mu(\varphi))/\mu(\psi)$ and $\phi:=\varphi-c_{\mu}\psi$, so that $\mu(\phi)=-h$. Since $\mu$ is ergodic with respect to the skew-product map, the Birkhoff ergodic theorem implies that \[\lim_{n\to\infty}\frac{S_{n}\psi_{\omega}(x)}{n}=\mu(\psi),\quad \lim_{n\to\infty}\frac{S_{n}\phi_{\omega}(x)}{n}=-h\] for $\mu$-a.e. $(\omega,x)\in J$. Fix $0<\xi<\min\{\mu(\psi),m\}$ and $0<a<1/4$. Since $\mu\in\M_{\prob}^{e}(f)$ and $h_{\mu}(f)\leq h_{\tp}(f)<\infty$, the Brin--Katok local entropy formula in \cite[Theorem 2.1]{zhu2} implies that \[\lim_{\epsilon\to0}\liminf_{n\to\infty}-\frac{1}{n}\log\mu_{\omega}(B_{n}^{\omega}(x,\epsilon))=h\] holds for $\mu$-a.e. $(\omega,x)\in J$, and allows us to choose a radius $\epsilon:=\epsilon(\xi,a)>0$ such that the set \[C_{\epsilon}:=\left\{(\omega,x)\in\Omega_{*}\times X\colon\liminf_{n\to\infty}-\frac{1}{n}\log\mu_{\omega}(B_{n}^{\omega}(x,\epsilon))\geq h-\frac{\xi}{2}\right\}\] satisfies $\mu(C_\epsilon)>1-a/2$. For each $N\geq1$, define 
    \begin{multline}\label{A_N}
        A_{N}:=\Biggl\{(\omega,x)\in C_{\epsilon}\colon\mu_{\omega}(B_{n}^{\omega}(x,\epsilon))\leq \exp(-n(h-\xi)), \\
        \bigg|\frac{S_{n}\psi_{\omega}(x)}{n}-\mu(\psi)\bigg|\leq\xi,\,\frac{S_{n}\phi_{\omega}(x)}{n}\geq-(h+\xi)\text{ for every }n\geq N\Biggr\}.
    \end{multline}
    The sequence of sets $\{A_{N}\}_{N\geq1}$ is increasing and $\mu(\bigcup_{N=1}^{\infty}A_{N})=\mu(C_\epsilon)>1-a/2$ up to a $\mu$-null set. Choose $N\geq1$ so that $\mu(A_N)>1-a$. By increasing $N$ if necessary, we obtain a measurable set $A$ with $\mu(A)>1-a$ on which the entropy and Birkhoff average estimates hold for every $n\geq N$. The maps involved are jointly measurable, so $\omega\mapsto \mu_{\omega}(A_\omega)$ is measurable. Hence $\Omega_{a}:=\{\omega\in\Omega_{*}\colon\mu_{\omega}(A_{\omega})\geq1/2\}$ satisfies $\prob(\Omega_{a})>1-2a$.

    For $(\omega,x)\in A$, let $k_{T}=k_{T}(\omega,x)$ be the induced time determined by $S_{k_{T}}\psi_{\omega}(x)\leq T<S_{k_{T}+1}\psi_{\omega}(x)$ for $T>0$. For sufficiently large $T$, the induced time $k_{T}$ is at least $N$ for every $(\omega,x)\in A$. Applying \eqref{A_N} at $k_{T}$ and $k_{T}+1$ gives \[\frac{T}{\mu(\psi)+\xi}-1< k_{T}\leq\frac{T}{\mu(\psi)-\xi}.\] Using $S_{k_{T}}\psi_{\omega}(x)>T-\supnorm{\psi}$ and $S_{k_{T}}\phi_{\omega}(x)\geq-k_{T}(h+\xi)$, we have the lower bound
    \begin{equation*}
        \begin{aligned}
            S_{k_{T}}\varphi_{\omega}(x)&=c_{\mu}S_{k_{T}}\psi_{\omega}(x)+S_{k_{T}}\phi_{\omega}(x) \\
            &\geq c_{\mu}T-|c_{\mu}|\supnorm{\psi}-\frac{h+\xi}{\mu(\psi)-\xi}T.
        \end{aligned}
    \end{equation*}
    
    Set $\ell_{T}:=\max\{1,\lfloor T/(\mu(\psi)+\xi)\rfloor-1\}$. It does not exceed $k_{T}$ and $\ell_{T}/T\to 1/(\mu(\psi)+\xi)$ as $T\to\infty$. For $\omega\in\Omega_{a}$, let $E_{\omega,T}\subseteq A_{\omega}$ be a maximal $(\omega,\ell_{T},\epsilon/2)$-separated set. Its Bowen balls of radius $\epsilon$ cover $A_{\omega}$, and \eqref{A_N} gives $\# E_{\omega,T}\geq\exp(\ell_{T}(h-\xi))/2$. Grouping these points by induced times produces admissible separated sets, since $k_{T}\geq \ell_{T}$. We then have 
    \begin{equation*}
        \begin{aligned}
            P_{\psi,T}(f,\omega,\varphi,\epsilon/2)&\geq\sum_{x\in E_{\omega,T}}\e^{S_{k_T}\varphi_{\omega}(x)} \\
            &\geq\frac{1}{2}\exp(\ell_{T}(h-\xi)+c_{\mu}T-|c_{\mu}|\supnorm{\psi}-\frac{h+\xi}{\mu(\psi)-\xi}T).
        \end{aligned}
    \end{equation*}
    Let $C_{T}(\omega):=T^{-1}\log P_{\psi,T}(f,\omega,\varphi,\epsilon/2)$. Then \[C_{T}(\omega)\geq\frac{\ell_{T}}{T}(h-\xi)+c_{\mu}-\frac{h+\xi}{\mu(\psi)-\xi}-\frac{1}{T}\left(|c_{\mu}|\supnorm{\psi}+\log2\right).\] Define \[r_{T}:=\ell_{T}-\frac{T}{\mu(\psi)+\xi}=\left\lfloor\frac{T}{\mu(\psi)+\xi}\right\rfloor-1-\frac{T}{\mu(\psi)+\xi}\] for sufficiently large $T$. Then $|r_{T}|\leq2$, hence \[\frac{\ell_{T}}{T}(h-\xi)=\frac{h-\xi}{\mu(\psi)+\xi}+\frac{r_{T}}{T}(h-\xi)\geq\frac{h-\xi}{\mu(\psi)+\xi}-\frac{2|h-\xi|}{T}.\] Define \[C(\xi):=c_{\mu}+\frac{h-\xi}{\mu(\psi)+\xi}-\frac{h+\xi}{\mu(\psi)-\xi}.\] It is clear that $\lim_{\xi\to0}C(\xi)=c_{\mu}$. Then \[C_{T}(\omega)\geq c_{\mu}+\frac{h-\xi}{\mu(\psi)+\xi}-\frac{h+\xi}{\mu(\psi)-\xi}-\frac{K_{\xi}}{T}=C(\xi)-\frac{K_{\xi}}{T},\] where $K_{\xi}:=|c_{\mu}|\supnorm{\psi}+\log2+2|h-\xi|$ only depends on $\xi$. It suffices to pick a sufficiently large $T_{0}$ such that $\ell_{T}\geq N$ for all $T>T_{0}$.
    
    For every $\omega\in\Omega_{*}$, a singleton in a nonempty time level gives \[S_{k}\varphi_{\omega}(x)\geq-\sum_{i=0}^{k-1}\supnorm{\varphi_{\theta^{i}\omega}}\geq-\sum_{i=0}^{L_{T}-1}\supnorm{\varphi_{\theta^{i}\omega}}\] where $L_{T}:=\max\{1,N_{T}\}$. Hence, \[C_{T}(\omega)\geq-\frac{1}{T}\sum_{i=0}^{L_{T}-1}\supnorm{\varphi_{\theta^{i}\omega}}\] for every $\omega\in\Omega_{*}$. Integrating both sides over $\Omega$ gives \[\int_{\Omega}C_{T}(\omega)\,\probinte\geq\prob(\Omega_{a})\left(C(\xi)-\frac{K_{\xi}}{T}\right)-\int_{\Omega_{a}^{c}}\frac{1}{T}\sum_{i=0}^{L_{T}-1}\supnorm{\varphi_{\theta^{i}\omega}}\,\probinte.\]  The ergodic averages of $\supnorm{\varphi_{\omega}}$ converge in $L^1$, and $L_{T}/T\to m^{-1}$. Consequently, their integrals over the fixed set $\Omega_{a}^{c}$ converge to $\prob(\Omega_{a}^{c})\onenorm{\varphi}/m$. Combining the bounds on $\Omega_{a}$ and its complement, taking the lower limit as $T\to\infty$ yields
    \begin{equation*}
        \begin{aligned}
            \liminf_{T\to\infty}\int_{\Omega}C_{T}(\omega)\,\probinte&\geq\prob(\Omega_{a})C(\xi)-\prob(\Omega_{a}^{c})\frac{\onenorm{\varphi}}{m} \\
            &=C(\xi)-\prob(\Omega_{a}^{c})\left(C(\xi)+\frac{\onenorm{\varphi}}{m}\right) \\
            &\geq C(\xi)-2a\max\left\{0,C(\xi)+\frac{\onenorm{\varphi}}{m}\right\}.
        \end{aligned}
    \end{equation*} Letting $a\to0$ and then $\xi\to0$, and taking the supremum over ergodic invariant measures completes the proof.
\end{proof}
\subsection{Critical exponents}\label{sec4.2}
We now define a weighted tail partition sum and its Poincar\'{e}-type critical exponent. For $T>0$, set \[Y_{\omega,n}=Y_{\omega,n}(T):=\{x\in X\colon S_{n}\psi_{\omega}(x)>T\}\] for each $n\in\N$. Define \[G_{\omega,T}:=\{n\in\N\colon Y_{\omega,n}(T)\neq\varnothing\}.\]
For $\omega\in\Omega_{*}$, let
\begin{multline*}
    R_{\psi,T}(f,\omega,\varphi,\epsilon):=\sup\biggl\{\sum_{n\in G_{\omega,T}}\sum_{x\in G_n}\e^{S_{n}\varphi_{\omega}(x)}\colon \\
    G_{n}\text{ is an $(\omega,n,\epsilon)$-separated set of }Y_{\omega,n},\,n\in G_{\omega,T}\biggr\},
\end{multline*}
and $R_{\psi,T}(f,\omega,\varphi,\epsilon)=1$ for $\omega\in\Omega\setminus\Omega_{*}$. We first verify that the tail partition sum is measurable for every fixed $T>0$ and $\epsilon>0$.
\begin{lemma}\label{meas2}
    For $\phi\in L_{\prob}^{1}(\Omega,C(X))$, the map $\omega\mapsto R_{\psi,T}(f,\omega,\phi,\epsilon)$ is measurable.
\end{lemma}
\begin{proof}
    For each $n\geq1$, let $r_{n,T}(f,\omega,\phi,\epsilon)$ be the separated partition sum on $Y_{\omega,n}(T)$ as in \eqref{separated}, with value zero when its set is empty. For every $k\geq1$ and $a>0$, let 
    \[A_{k}^{n,a}:=\biggl\{(\omega,x_{1},\ldots,x_{k})\in\Omega_{*}\times X^{k}\colon S_{n}\psi_{\omega}(x_i)>T, d_{n}^{\omega}(x_{i},x_{j})>\epsilon, i\neq j,\sum_{i=1}^{k}e^{S_{n}\phi_{\omega}(x_i)}>a\biggr\}.\]
    By the joint measurability of Birkhoff sums and Bowen metrics, $A_{k}^{n,a}\in\F_{*}\otimes\B^{k}$. By \lemref{meas}, $\{\omega\in\Omega_{*}\colon r_{n,T}(f,\omega,\phi,\epsilon)>a\}=\bigcup_{k\in\N}\pi_{\Omega}(A_{k}^{n,a})\in\F_{*}$. Each truncated sum is measurable by the measurable projection theorem. Since $R_{\psi,T}(f,\omega,\phi,\epsilon)=\sum_{n\geq1}r_{n,T}(f,\omega,\phi,\epsilon)$ on $\Omega_{*}$, the assertion follows by taking the increasing limit of partial sums and using the prescribed extension outside $\Omega_{*}$.
\end{proof}
We define the \emph{Poincar\'{e}-type critical exponent} \[\beta_{c}:=\inf\left\{\beta\in\R\colon\lim_{\epsilon\to0}\limsup_{T\to\infty}R_{\psi,T}(f,\omega,\varphi-\beta\psi,\epsilon)<\infty\text{ for }\almost\right\}\] with the convention $\inf\varnothing=\infty$. We will show that this critical exponent coincides with the threshold of classical fiber topological pressure, including the case of infinite fiber topological entropy.
\begin{proposition}\label{eq}
    The critical exponent $\beta_{c}$ coincides with $\beta_{0}$.
\end{proposition}
\begin{proof}
    First suppose $h_{\tp}(f)<\infty$. Fix $\omega\in\Omega_{*}$ and $T>0$. Since $Y_{\omega,n}(T)=X$ for $n>T/m$ and $Y_{\omega,n}(T)=\varnothing$ for $n\leq T/\supnorm{\psi}$, we have \[\sum_{n>T/m}P_{n}(f,\omega,\varphi-\beta\psi,\epsilon)\leq R_{\psi,T}(f,\omega,\varphi-\beta\psi,\epsilon)\leq\sum_{n\geq T/\supnorm{\psi}}P_{n}(f,\omega,\varphi-\beta\psi,\epsilon).\]

    Fix $\beta>\beta_{0}$ and put $\pres(\beta)=\pres(f,\varphi-\beta\psi)=p<0$. \thmref{randompres} gives exponential decay of the classical partition sums at each fixed radius $\epsilon$, \[\limsup_{n\to\infty}\frac{1}{n}\log P_{n}(f,\omega,\varphi-\beta\psi,\epsilon)\leq p\] for each fixed $\epsilon>0$ since the partition sum is nondecreasing as $\epsilon\to0$. There exists $N:=N(\omega,\epsilon,\beta)\geq1$ such that \[P_{n}(f,\omega,\varphi-\beta\psi,\epsilon)\leq \e^{\frac{p}{2}n}\] for all $n\geq N$. Take $T>0$ so large that $T/\supnorm{\psi}\geq N$. Then \[R_{\psi,T}(f,\omega,\varphi-\beta\psi,\epsilon)\leq\sum_{n\geq T/\supnorm{\psi}}P_{n}(f,\omega,\varphi-\beta\psi,\epsilon)\leq\sum_{n\geq T/\supnorm{\psi}}\e^{\frac{p}{2}n}.\] Consequently, the upper tail bound tends to zero as $T\to\infty$, so letting $\epsilon\to0$ gives \[\lim_{\epsilon\to0}\limsup_{T\to\infty}R_{\psi,T}(f,\omega,\varphi-\beta\psi,\epsilon)=0\] for $\almost$. It follows that the quantity in the definition of $\beta_{c}$ vanishes for almost every $\omega$, so $\beta_{0}\geq\beta_{c}$.

    Conversely, fix $\beta<\beta_{0}$ and put $\pres(\beta)=q>0$. For almost every $\omega$, there is a radius $\epsilon=\epsilon_{\omega}>0$ at which the classical partition sums have positive upper exponential growth, in the sense that \[\limsup_{n\to\infty}\frac{1}{n}\log P_{n}(f,\omega,\varphi-\beta\psi,\epsilon)>\frac{q}{2}.\] Hence, there exists a subsequence $\{n_{j}\}_{j\geq1}$ tending to infinity such that $P_{n_j}(f,\omega,\varphi-\beta\psi,\epsilon)\geq \e^{qn_j/2}$. The corresponding series therefore diverges, and the lower tail estimate gives $R_{\psi,T}(f,\omega,\varphi-\beta\psi,\epsilon_{\omega})=\infty$ for all $T>0$. By monotonicity, the sum remains infinite at every smaller radius $\epsilon$. Thus $\beta_{0}\leq\beta_{c}$.
    
    If $h_{\tp}(f)=\infty$, then $\pres(\beta)=\infty$ for every fixed $\beta\in\R$. The same argument excludes every real $\beta$ from the defining set, giving $\beta_{c}=\beta_{0}=\infty$.
\end{proof}
Fix $0<\delta\leq m$, and define the Laplace-type partition series by \[L_{s}(\omega,\epsilon):=\sum_{j\geq1}\e^{-sj\delta}P_{\psi,j\delta}(f,\omega,\varphi,\epsilon).\] Let $s_{c}(\epsilon):=\inf\{s\in\R\colon L_{s}(\omega,\epsilon)<\infty\text{ for }\almost\}$ with the convention $\inf\varnothing=\infty$ and let \[s_{c}:=\lim_{\epsilon\to0}s_{c}(\epsilon)\] be the \emph{Laplace-type critical exponent}. For convenience, for each $n\geq1$ define \[P_{n,j}:=\sup\left\{\sum_{x\in E}\e^{S_{n}\varphi_{\omega}(x)}\colon E\text{ is an $(\omega,n,\epsilon)$-separated set of }X_{\omega,n}(j\delta)\right\}\] and $P_{n,j}=0$ whenever $X_{\omega,n}(j\delta)=\varnothing$. Then $P_{\psi,T}(f,\omega,n,\epsilon)=\sum_{n\geq1}P_{n,j}+\ind_{\{S_{\omega,j\delta}=\varnothing\}}$. The following comparison links $L_{s}(\omega,\epsilon)$ with classical fiber partition sums.
\begin{lemma}\label{estimate}
    For $\omega\in\Omega_{*}$, $\epsilon>0$ and $s\in\R$, we have 

    (1) \begin{equation*}
            L_{s}(\omega,\epsilon)\leq\e^{|s|\supnorm{\psi}}\sum_{n\geq1}\left(1+\frac{(n+1)\supnorm{\psi}}{\delta}\right)P_{n}(f,\omega,\varphi-s\psi,\epsilon)+\sum_{j\geq1}\e^{-sj\delta}\ind_{\{S_{\omega,j\delta}=\varnothing\}}
        \end{equation*}

     (2) \[L_{s}(\omega,\epsilon)\geq\e^{-|s|\supnorm{\psi}}\sum_{n\geq1}P_{n}(f,\omega,\varphi-s\psi,\epsilon)+\sum_{j\geq1}\e^{-sj\delta}\ind_{\{S_{\omega,j\delta}=\varnothing\}}.\]  
\end{lemma}
\begin{proof}
    The set $S_{\omega,j\delta}$ is nonempty provided $j\delta\geq\supnorm{\psi}$. Then \[\sum_{j\geq1}\e^{-sj\delta}\ind_{\{S_{\omega,j\delta}=\varnothing\}}\leq\sum_{j\geq1,j\delta<\supnorm{\psi}}\e^{-sj\delta}<\infty.\] The induced time set $S_{\omega,j\delta}$ is finite for each $j$, and each $\e^{-sj\delta}P_{n,j}$ is nonnegative, we may rewrite the series $L_{s}(\omega,\epsilon)$, obtaining \[L_{s}(\omega,\epsilon)=\sum_{n\geq1}\sum_{j\geq1}\e^{-sj\delta}P_{n,j}+\sum_{j\geq1}\e^{-sj\delta}\ind_{\{S_{\omega,j\delta}=\varnothing\}}.\] Since the second series is always finite, it suffices to estimate the first one to prove the lemma. Whenever $x\in X_{\omega,n}(j\delta)$, we have $S_{n}\psi_{\omega}(x)\leq j\delta<S_{n+1}\psi_{\omega}(x)$, hence $0\leq j\delta-S_{n}\psi_{\omega}(x)<\supnorm{\psi}$ and 
    \begin{equation}\label{expcompare}
        \e^{-|s|\supnorm{\psi}}\e^{S_{n}\varphi_{\omega}(x)-sS_{n}\psi_{\omega}(x)}\leq \e^{-sj\delta}\e^{S_{n}\varphi_{\omega}(x)}\leq \e^{|s|\supnorm{\psi}}\e^{S_{n}\varphi_{\omega}(x)-sS_{n}\psi_{\omega}(x)}.
    \end{equation} For fixed $n$, at most $1+\delta^{-1}(n+1)\supnorm{\psi}$ indices $j$ can contribute because $X_{\omega,n}(j\delta)\neq\varnothing$ if $nm\leq j\delta<(n+1)\supnorm{\psi}$. Therefore, for each $n\geq1$ we have
    \[\sum_{j\geq1}\e^{-sj\delta}P_{n,j}\leq\e^{|s|\supnorm{\psi}}\left(1+\frac{(n+1)\supnorm{\psi}}{\delta}\right)P_{n}(f,\omega,\varphi-s\psi,\epsilon).\]
    Summation over $n\geq1$ gives the upper bound (i).

    For the lower bound, fix an $(\omega,n,\epsilon)$-separated set $E\subseteq X$ and assign to each point $x\in E$ the index $j(x):=\lceil S_{n}\psi_{\omega}(x)/\delta\rceil\geq1$. Then $x\in X_{\omega,n}(j(x)\delta)$ by $S_{n}\psi_{\omega}(x)\leq j(x)\delta<S_{n}\psi_{\omega}(x)+\delta\leq S_{n+1}\psi_{\omega}(x)$. The resulting partition of $E$ consists of separated subsets of the appropriate induced time sets. Applying the preceding exponential comparison in \eqref{expcompare} gives \[\sum_{j\geq1}\e^{-sj\delta}P_{n,j}\geq\sum_{x\in E}\e^{-sj(x)\delta}\e^{S_{n}\varphi_{\omega}(x)}\geq\e^{-|s|\supnorm{\psi}}\sum_{x\in E}\e^{S_{n}\varphi_{\omega}(x)-sS_{n}\psi_{\omega}(x)}.\] Taking the supremum gives the lower bound (ii).
\end{proof}
\begin{proposition}\label{laplace}
    The exponent $s_{c}$ equals the classical pressure threshold $\beta_{0}$.
\end{proposition}
\begin{proof}
    The series $L_{s}(\omega,\epsilon)$ is nonincreasing in $s$ and nondecreasing as $\epsilon\to0$. Consequently, $s_{c}(\epsilon)$ is nondecreasing as $\epsilon\to0$, so the limit defining $s_{c}$ exists.

    Suppose $h_{\tp}(f)<\infty$. Define $\pres(s):=\pres(f,\varphi-s\psi)$. Fix $s>\beta_{0}$. Since $\pres(s)<0$, \thmref{randompres} implies the exponential decay of the classical partition sums at every fixed radius for almost every $\omega$; $\limsup_{n\to\infty}n^{-1}\log P_{n}(f,\omega,\varphi-s\psi,\epsilon)\leq\pres(s)<0$. Then there exists $N\geq1$ such that $P_{n}(f,\omega,\varphi-s\psi,\epsilon)\leq\e^{\pres(s)n/2}$ for all $n\geq N$. Hence, \[\sum_{n\geq1}\left(1+\frac{(n+1)\supnorm{\psi}}{\delta}\right)P_{n}(f,\omega,\varphi-s\psi,\epsilon)<\infty.\] \lemref{estimate} (i) then gives $L_{s}(\omega,\epsilon)<\infty$ almost surely. Hence $s_{c}(\epsilon)\leq s$ for every $\epsilon>0$, and letting $s\to\beta_{0}$ yields $s_{c}\leq\beta_{0}$. Fix $s<\beta_{0}$ and choose a sequence $\{\epsilon_{k}\}_{k\geq1}$ tending to $0$. The positive pressure $\pres(s)>0$ and \lemref{estimate} (ii) imply $\prob(\bigcup_{k\geq1}\{\omega\colon L_{s}(\omega,\epsilon_k)=\infty\})=1$. Hence this event has positive probability for at least one fixed $k$. For that $k$, $L_{t}(\omega,\epsilon_{k})=\infty$ on a set of positive probability for every $t\leq s$. No such $t$ satisfies the almost-sure finiteness condition defining $s_{c}(\epsilon_{k})$. Therefore $s_{c}(\epsilon_{k})\geq s$, and monotonicity gives $s_{c}\geq s$. Letting $s\to\beta_{0}$ proves the reverse inequality. The same argument applies to every $s\in\R$ in the infinite entropy case, giving $s_{c}=\beta_{0}=\infty$.
\end{proof}
\propref{ineq1}, \propref{ineq2}, \propref{eq} and \propref{laplace} prove the main characterization of induced fiber pressure in the finite entropy case. \lemref{presentropy} and the infinite entropy parts of \propref{eq} and \propref{laplace} give the remaining case.
\begin{theorem}\label{critical}
    Under the standing assumptions, the induced fiber topological pressure, the classical pressure threshold, and the two critical exponents coincide: \[\pres_{\psi}(f,\varphi)=\beta_{0}=\beta_{c}=s_{c}.\]
\end{theorem}
When the classical fiber topological entropy is finite, this characterization identifies the induced pressure with the unique zero of the finite classical fiber pressure function.
\begin{corollary}\label{zero}
    If $h_{\tp}(f)<\infty$, then $\pres(f,\varphi-\pres_{\psi}(f,\varphi)\psi)=0$.
\end{corollary}
We conclude this section with an almost-sure fiberwise characterization of induced pressure.
\begin{theorem}\label{almostsure}
    For $\almost$, \[\pres_{\psi}(f,\varphi)=\lim_{\epsilon\to0}\limsup_{T\to\infty}\frac{1}{T}\log P_{\psi,T}(f,\omega,\varphi,\epsilon)=\lim_{\epsilon\to0}\limsup_{T\to\infty}\frac{1}{T}\log Q_{\psi,T}(f,\omega,\varphi,\epsilon).\]
\end{theorem}
\begin{proof}
    Suppose first that $h_{\tp}(f)<\infty$, and fix $\beta>\beta_{0}=\pres_{\psi}(f,\varphi)$. Put $\pres(\beta)=\pres(f,\varphi-\beta\psi)=p<0$. \thmref{randompres} provides $P_{n}(f,\omega,\varphi-\beta\psi,\epsilon)\leq \e^{pn/2}$ for almost every $\omega$ and each fixed $\epsilon>0$, for all sufficiently large $n$. The estimates \eqref{expupperbound1} and \eqref{expupperbound2} used in \propref{ineq1} give \[P_{\psi,T}(f,\omega,\varphi,\epsilon)\leq \e^{\beta T+|\beta|\supnorm{\psi}}\sum_{n\in I_{T}}P_{n}(f,\omega,\varphi-\beta\psi,\epsilon)\leq C_{p}\e^{\beta T+|\beta|\supnorm{\psi}}\e^{\frac{pT}{2\supnorm{\psi}}}\] for some constant $C_{p}>0$, leading to \[\limsup_{T\to\infty}\frac{1}{T}\log P_{\psi,T}(f,\omega,\varphi,\epsilon)\leq\beta.\] Taking countable sequences $\{\epsilon_{j}\}_{j\geq1}$ tending to zero and $\{\beta_{k}\}_{k\geq1}$ decreasing to $\beta_{0}$ and intersecting full-measure sets proves the asserted upper bound \[\lim_{\epsilon\to0}\limsup_{T\to\infty}\frac{1}{T}\log P_{\psi,T}(f,\omega,\varphi,\epsilon)\leq\beta_{0}\] almost surely.

    Fix an ergodic invariant measure $\mu\in\M_{\prob}^{e}(f)$. \propref{ineq2} shows that, for every sufficiently small $\xi>0$ and every $a\in(0,1/4)$, there are $\epsilon_{\xi,a}>0$ and a measurable subset $\Omega_{\xi,a}$ of measure at least $1-2a$ such that \[\liminf_{T\to\infty}\frac{1}{T}\log P_{\psi,T}(f,\omega,\varphi,\epsilon_{\xi,a}/2)\geq C(\xi)\] for every $\omega\in\Omega_{\xi,a}$. We have \[\lim_{\epsilon\to0}\liminf_{T\to\infty}\frac{1}{T}\log P_{\psi,T}(f,\omega,\varphi,\epsilon)\geq C(\xi)\] by monotonicity as $\epsilon\to0$. Choose $\{a_{k}\}_{k\geq1}$ tending to zero. Since the union $E_{\mu,\xi}:=\bigcup_{k\geq1}\Omega_{\xi,a_{k}}$ has full measure, the limiting lower bound holds on $E_{\mu,\xi}$. A further countable intersection over $\{\xi_{\ell}\}_{\ell\geq1}$ tending to zero yields a full-measure set $E_\mu$ on which
    \[\lim_{\epsilon\to0}\liminf_{T\to\infty}\frac{1}{T}\log P_{\psi,T}(f,\omega,\varphi,\epsilon)\geq\frac{h_{\mu}(f)+\int_{J}\varphi\,\dd\mu}{\int_{J}\psi\,\dd\mu}\] for every $\omega\in E_{\mu}$. Finally, choose a sequence of ergodic measures whose ratios approach $\beta_{0}$ and intersect their associated full-measure sets. This proves the lower bound. The spanning formula follows from \eqref{comparison} and the ergodic theorem applied to $V_{T}(\omega,\epsilon)$ in \propref{equivalent}.

    \thmref{randompres} for the zero potential makes the right-hand side of \eqref{entropylowerbound} diverge as $\epsilon\to0$. Thus the separated pressure is infinite almost surely. The inequality \eqref{comparison}, together with the ergodic estimate for $V_{T}(\omega,\epsilon)$, gives the same conclusion for spanning sums. \lemref{presentropy} completes the proof.
\end{proof}
\begin{remark}
    For $\psi=\ind$, \thmref{almostsure} reduces to the almost-sure fiberwise pressure formula of \thmref{randompres} from Y. Kifer.
\end{remark}
\section{Variational principle and equilibrium states}\label{sec5}
The pressure characterization obtained in \secref{sec4} yields a variational principle and identifies induced equilibrium states with classical equilibrium states for a normalized potential.
\begin{theorem}\label{randominducedvp}
    \[\pres_{\psi}(f,\varphi)=\sup\left\{\frac{h_{\mu}(f)+\int_{J}\varphi\,\dd\mu}{\int_{J}\psi\,\dd\mu}\colon\mu\in\M_{\prob}(f)\right\}.\]
\end{theorem}
\begin{proof}
    If $h_{\tp}(f)=\infty$, both sides are infinite by \lemref{presentropy} and the classical variational principle. Suppose that $\pres(f,\varphi-\beta\psi)<\infty$ for all $\beta\in\R$. Then $\pres_{\psi}(f,\varphi)<\infty$. Since $m\leq\int_{J}\psi\,\dd\mu\leq\supnorm{\psi}$, the classical variational principle gives \[\pres(f,\varphi-\beta\psi)=\sup\left\{\left(\frac{h_{\mu}(f)+\int_{J}\varphi\,\dd\mu}{\int_{J}\psi\,\dd\mu}-\beta\right)\int_{J}\psi\,\dd\mu\colon\mu\in\M_{\prob}(f)\right\}.\] Thus $\pres(f,\varphi-\beta\psi)\leq0$ if and only if $\beta$ is at least the supremum of the ratios. \thmref{critical} proves the assertion.
\end{proof}
We state some properties of induced fiber pressure.
\begin{proposition}
    Let $\varphi_1,\varphi_2\in L_{\prob}^{1}(\Omega,C(X)),\psi\in L_{\prob}^{\infty}(\Omega,C(X))$ with $\essinf_{\omega\in\Omega}\inf_{x\in X}\psi_{\omega}(x)>0$. The $\psi$-induced fiber pressure satisfies the following properties.

    (1) If $\varphi_{1}\leq\varphi_{2}$ for $\almost$ and all $x\in X$, then $\pres_{\psi}(f,\varphi_1)\leq\pres_{\psi}(f,\varphi_2)$.

    (2) If $h_{\tp}(f)<\infty$, then the $\psi$-induced fiber pressure function is Lipschitz continuous with respect to the $L^1$-norm.

    (3) $\pres_{\psi}(f,\varphi_{1}+\varphi_{2})\leq\pres_{\psi}(f,\varphi_1)+\pres_{\psi}(f,\varphi_2)$.

    (4) $\pres_{\psi}(f,t\varphi_{1}+(1-t)\varphi_{2})\leq t\pres_{\psi}(f,\varphi_{1})+(1-t)\pres_{\psi}(f,\varphi_{2})$ for all $t\in(0,1)$.

    (5) $\pres_{\psi}(f,\varphi_1+\varphi_{2}\circ\Theta-\varphi_2)=\pres_{\psi}(f,\varphi_1)$.
\end{proposition}
\begin{proof}
    Properties (1), (4) and (5) follow from \thmref{randominducedvp}, using invariance for (5). For (2), each ratio changes by at most $m^{-1}\onenorm{\varphi_{1}-\varphi_{2}}$, so \[|\pres_{\psi}(f,\varphi_1)-\pres_{\psi}(f,\varphi_2)|\leq m^{-1}\onenorm{\varphi_{1}-\varphi_{2}},\] where the Lipschitz constant is $m^{-1}>0$.

    For (3), the infinite entropy case is immediate. In the finite entropy case, 
    \begin{equation*}
        \begin{aligned}
            \frac{h_{\mu}(f)+\int_{J}(\varphi_{1}+\varphi_{2})\,\dd\mu}{\int_{J}\psi\,\dd\mu}&=\frac{h_{\mu}(f)+\int_{J}\varphi_{1}\,\dd\mu}{\int_{J}\psi\,\dd\mu}+\frac{h_{\mu}(f)+\int_{J}\varphi_{2}\,\dd\mu}{\int_{J}\psi\,\dd\mu}-\frac{h_{\mu}(f)}{\int_{J}\psi\,\dd\mu} \\
            &\leq\pres_{\psi}(f,\varphi_{1})+\pres_{\psi}(f,\varphi_{2})
        \end{aligned}
    \end{equation*} since the last fraction is nonnegative. Taking the supremum completes the proof.
\end{proof}
\begin{definition}
    Assume $h_{\tp}(f)<\infty$. A $\Theta$-invariant Borel probability measure $\mu\in\M_{\prob}(f)$ is an \emph{induced equilibrium state} for $(\varphi,\psi)$ if 
    \[\pres_{\psi}(f,\varphi)=\frac{h_{\mu}(f)+\int_{J}\varphi\,\dd\mu}{\int_{J}\psi\,\dd\mu}.\]
    Denote the set of such measures by $\M_{\prob,(\varphi,\psi)}(f)$.
\end{definition}
\begin{proposition}\label{same}
    Assume $h_{\tp}(f)<\infty$. Let $\varphi\in L_{\prob}^{1}(\Omega,C(X))$, and let $\psi\in L_{\prob}^{\infty}(\Omega,C(X))$ with $\essinf_{\omega\in\Omega}\inf_{x\in X}\psi_{\omega}(x)>0$. Then $\M_{\prob,(\varphi,\psi)}(f)=\E_{\prob}(\varphi-\pres_{\psi}(f,\varphi)\psi,f)$.
\end{proposition}
\begin{proof}
    By \corref{zero}, the potential $\varphi-\pres_{\psi}(f,\varphi)\psi$ has classical fiber pressure zero. The defining equality for an induced equilibrium state is therefore exactly the equality defining an equilibrium state for this potential.
\end{proof}
The set $\M_{\prob}(f)$ is nonempty and compact in the weak$^*$ topology induced by the duality with $L_{\prob}^{1}(\Omega,C(X))$ by \cite[Lemma 2.1]{kifer1}. In this topology, $\mu_n\to\mu$ means $\int_J\varphi\,\dd\mu_n\to\int_J\varphi\,\dd\mu$ as $n\to\infty$ for each $\varphi\in L_{\prob}^{1}(\Omega,C(X))$. The next result assumes upper semicontinuity of $\mu\mapsto h_{\mu}(f)$ with respect to this topology.
\begin{corollary}
    If, in addition, the entropy map $\mu\mapsto h_{\mu}(f)$ is upper semicontinuous on $\M_{\prob}(f)$, then $\M_{\prob,(\varphi,\psi)}(f)$ is a nonempty compact subset of $\M_{\prob}(f)$. Furthermore, the extreme points of $\M_{\prob,(\varphi,\psi)}(f)$ are precisely the ergodic measures in $\M_{\prob,(\varphi,\psi)}(f)$.
\end{corollary}
\begin{proof}
    The functional \[D(\mu):=h_{\mu}(f)+\int_{J}(\varphi-\pres_{\psi}(f,\varphi)\psi)\,\dd\mu\] is affine and upper semicontinuous on the compact convex set $\M_{\prob}(f)$. The maximum of $D(\mu)$ is zero by the classical variational principle. Its maximizers form a nonempty compact face of $\M_{\prob}(f)$, identified with $\M_{\prob,(\varphi,\psi)}(f)$ by \propref{same}. The extreme points of this face are precisely its ergodic measures \cite[Lemma 6.19]{crauel}. The Krein--Milman theorem therefore yields an ergodic induced equilibrium state for $(\varphi,\psi)$.
\end{proof}

\section{Comparison with skew-product induced pressure}\label{sec6}
Assume additionally that $\Omega$ is a compact metric space, that $\theta$ and $\Theta$ are continuous, and that $h_{\tp}(\Theta)<\infty$. Let $\varphi,\psi\in C(J,\R)$ with $\psi>0$. The $\psi$-induced topological pressure of $\varphi$ with respect to $\Theta$ is given by \[\pres_{\psi}(\Theta,\varphi)=\sup_{\mu\in\M(J,\Theta)}\frac{h_{\mu}(\Theta)+\int_{J}\varphi\,\dd\mu}{\int_{J}\psi\,\dd\mu},\] where $\M(J,\Theta)$ is the set of all $\Theta$-invariant Borel probability measures on $J$. For each $\prob\in\M(\Omega,\theta)$, define the \emph{$\psi$-induced pressure} relative to marginal $\prob$ by 
\begin{equation}\label{inducedrelative}
    \pres_{\psi,\prob}(\Theta,\varphi):=\sup_{\mu\in\M_{\prob}(f)}\frac{h_{\mu}(\Theta)+\int_{J}\varphi\,\dd\mu}{\int_{J}\psi\,\dd\mu}.
\end{equation}
Clearly $\pres_{\psi,\prob}(\Theta,\varphi)\leq\pres_{\psi}(\Theta,\varphi)$. The global pressure is recovered by taking the supremum over invariant, or equivalently ergodic base measures.
\begin{proposition}
    \[\pres_{\psi}(\Theta,\varphi)=\sup_{\prob\in\M(\Omega,\theta)}\pres_{\psi,\prob}(\Theta,\varphi)=\sup_{\prob\in\M^{e}(\Omega,\theta)}\pres_{\psi,\prob}(\Theta,\varphi).\]
\end{proposition}
\begin{proof}
    Every $\Theta$-invariant measure has a $\theta$-invariant marginal, so the first equality follows by grouping invariant measures according to their marginals. For the second, ergodic decomposition expresses each ratio as an average of the corresponding ergodic ratios, weighted by the positive integrals of $\psi$. Thus the supremum may be restricted to ergodic measures, whose marginals are ergodic.
\end{proof}
We attach the marginal $\prob$ to fiber pressures relative to it. Note that $\pres_{\prob}(f,\varphi)=\pres_{\ind,\prob}(f,\varphi)$. By the definition \eqref{inducedrelative}), \[\pres_{\ind,\prob}(\Theta,\varphi)=\sup\left\{h_{\mu}(\Theta)+\int_{J}\varphi\,\dd\mu\colon\mu\in\M_{\prob}(f)\right\}.\] We obtain the first equality connecting the induced pressure relative to $\prob$ and the fiber topological pressure.
\begin{proposition}
    \[\pres_{\ind,\prob}(\Theta,\varphi)=h_{\prob}(\theta)+\pres_{\ind,\prob}(f,\varphi).\]
\end{proposition}
\begin{proof}
    Since the marginal $\prob$ is fixed, $h_{\prob}(\theta)$ is constant in the variational principle. Applying the Abramov--Rokhlin formula $h_{\mu}(\Theta)=h_{\prob}(\theta)+h_{\mu}(f)$ in \cite[Corollary]{bogen2} proves the identity.
\end{proof}
Now fix $\prob\in\M^{e}(\Omega,\theta)$. By the Abramov--Rokhlin formula and \eqref{inducedrelative}, $\beta=\pres_{\psi,\prob}(\Theta,\varphi)$ is characterized by
\begin{equation}\label{observation}
    \pres_{\prob}(f,\varphi-\beta\psi)=-h_{\prob}(\theta)
\end{equation}
In contrast, $\pres_{\psi,\prob}(f,\varphi)$ is the zero of this pressure function. In particular, they coincide when $h_{\prob}(\theta)=0$.

Fig.~\ref{fig} illustrates the two values. The values $\beta_{1}:=\pres_{\psi,\prob}(f,\varphi)$ and $\beta_{2}:=\pres_{\psi,\prob}(\Theta,\varphi)$ correspond to the levels $\pres(\beta_{1})=0$ and $\pres(\beta_{2})=-h_{\prob}(\theta)\leq0$, respectively.

\begin{center}
\begin{minipage}{0.85\textwidth}
\centering
\begin{tikzpicture}[scale=0.8,>=stealth]

\draw[->] (-1,0) -- (7,0) node[below right] {$\beta$};
\draw[->] (0,-3) -- (0,3) node[above left] {$\pres(\beta)$};

\draw[dashed] (0,-1.6) -- (6.8,-1.6);
\def\ba{2.5}
\def\bb{5}
\def\hh{1.6}
\def\kk{0.25}
\fill[blue!5]
    (\ba,0) rectangle (\bb,-\hh);
\draw[thick, blue, domain=0.5:6.5, samples=150, smooth]
  plot (\x, {
    1.6*(exp(-0.3*(\x-2.5))-1)/
    (1-exp(-0.75))
  });

\filldraw[blue] (2.5,0) circle (1.8pt);
\filldraw[blue] (5,-1.6) circle (1.8pt);
\filldraw[black] (5,0) circle (1.8pt);
\filldraw[black] (0,-1.6) circle (1.8pt);

\draw[dashed] (5,0) -- (5,-1.6);
\draw[dashed] (2.5,0) -- (2.5,-1.6);

\node[left] at (0,-1.6) {$-h_{\mathbb P}(\theta)$};
\node[below left] at (2.5,0) {$\beta_1$};
\node[below right] at (5,0) {$\beta_2$};
\node[left] at (-1,0) {$\pres(\beta)=0$};

\end{tikzpicture}
\captionof{figure}{
    Graph of the pressure function $\beta\mapsto\pres(\beta)$.
}
\label{fig}
\end{minipage}
\end{center}

Motivated by \cite[Corollary II.8]{rahimi4}, the slope bounds in \propref{preszero} yield the following comparison.
\begin{corollary}
    \[\pres_{\psi,\prob}(f,\varphi)+\frac{h_{\prob}(\theta)}{\max_{J}\psi}\leq\pres_{\psi,\prob}(\Theta,\varphi)\leq\pres_{\psi,\prob}(f,\varphi)+\frac{h_{\prob}(\theta)}{\min_{J}\psi}.\]
\end{corollary}

\section{A symbolic example}\label{sec7}
We illustrate the pressure characterization with a symbolic example in which the scaling function depends on both the base and fiber coordinates.

Let $\Omega=\{0,1\}^{\Z}$ carry the two-sided Bernoulli measure $\prob=(\delta_{0}/2+\delta_{1}/2)^{\otimes\Z}$, and let $\theta\colon\Omega\to\Omega$ be the left shift. Let $X=\{0,1\}^{\N_0}$ with the symbolic metric \[d_{r}(x,y):=r^{\min\{i\geq0\colon x_{i}\neq y_{i}\}},\] where $r\in(0,1)$. Write $J=\Omega\times X$. Set $f_{\omega}=\sigma$ for every $\omega\in\Omega$, where $\sigma\colon X\to X$ is the left shift. Take \[\varphi\equiv0,\quad\psi_{\omega}(x)=1+\omega_{0}+x_{0}.\] Then $1\leq\psi_{\omega}(x)\leq3$, and \[S_{n}\psi_{\omega}(x)=n+\sum_{i=0}^{n-1}\omega_{i}+\sum_{i=0}^{n-1}x_{i}.\]

We begin with the classical fiber pressure. Fix $0<\epsilon<r$, and choose $\ell\geq1$ such that $r^{\ell+1}\leq\epsilon<r^{\ell}$. Distinct words of length $n+\ell$ give exactly the $(\omega,n,\epsilon)$-separated cylinders. Since $S_{n}\psi_{\omega}$ depends only on the first $n$ coordinates, we obtain
\begin{equation*}
    \begin{aligned}
        P_{n}(f,\omega,\varphi-\beta\psi,\epsilon)&=\sum_{x\in\{0,1\}^{n+\ell}}\e^{-\beta S_{n}\psi_{\omega}(x)} \\
        &=2^{\ell}\e^{-\beta\left(n+\sum_{i=0}^{n-1}\omega_{i}\right)}\sum_{x\in\{0,1\}^{n}}\e^{-\beta\sum_{i=0}^{n-1}x_{i}} \\
        &=2^{\ell}\e^{-\beta\left(n+\sum_{i=0}^{n-1}\omega_{i}\right)}(1+\e^{-\beta})^{n}.
    \end{aligned}
\end{equation*} Taking logarithms and integrating with respect to $\prob$ gives \[\pres(f,\varphi-\beta\psi)=-\frac{3}{2}\beta+\log(1+\e^{-\beta})\] in the limit as $n\to\infty$. Let $\beta_{0}>0$ be its unique zero.

We next evaluate the induced fiber pressure directly. Let $|x|$ denote the length of word $x\in X$. For $T\geq3$, let $\mathcal{W}_{T}(\omega)$ be the finite words $u$ satisfying \[S_{|u|-1}\psi_{\omega}(u)\leq T<S_{|u|}\psi_{\omega}(u),\] and put $n_{T}(\omega):=\#\mathcal{W}_{T}(\omega)$. These sums are well-defined on words because $\psi$ depends only on the current coordinates. The corresponding cylinders partition $X$. Every such word has length $k$ with $T/3<k\leq T+1$, and $T<S_{k}\psi_{\omega}(u)\leq T+3$.

Each stopping word of length $n+1$ determines $2^{\ell-1}$ cylinders of length $n+\ell$, so $\pres_{\psi,T}(f,\omega,\varphi,\epsilon)=2^{\ell-1}n_{T}(\omega)$. Set \[q:=\frac{\e^{-\beta_{0}}}{1+\e^{-\beta_{0}}},\] and let $\nu_{q}$ be the Bernoulli measure on $X$ with weights $1-q$ and $q$. For every word $x$ of length $k$, \[\nu_{q}([x])=\frac{\exp(-\beta_{0}\sum_{i=0}^{k-1}x_i)}{(1+\e^{-\beta_{0}})^{k}}.\] The equation defining $\beta_{0}$ yields $\e^{-\beta_{0}S_{k}\psi_{\omega}(x)}=\nu_{q}([x])\e^{H_{k}(\omega)}$, where $H_{k}(\omega)=\beta_{0}(k/2-\sum_{i=0}^{k-1}\omega_{i})$. By the Birkhoff ergodic theorem, $\lim_{k\to\infty}k^{-1}H_{k}(\omega)=0$ for $\almost$.

Consider the weighted sum \[Z_{T}(\omega):=\sum_{x\in\mathcal{W}_{T}(\omega)}\e^{-\beta_{0}S_{|x|}\psi_{\omega}(x)}.\] Since these cylinders partition $X$, $\sum_{x\in\mathcal{W}_{T}(\omega)}\nu_{q}([x])=1$. By the Birkhoff ergodic theorem, $k^{-1}H_{k}(\omega)\to0$ almost surely, and hence $T^{-1}\max_{1\leq k\leq\lfloor T\rfloor+1}|H_{k}(\omega)|\to0$ as $T\to\infty$. Since the stopping cylinders partition $X$, \[\e^{-\max_{1\leq k\leq\lfloor T\rfloor+1}|H_{k}(\omega)|}\leq Z_{T}(\omega)\leq\e^{\max_{1\leq k\leq\lfloor T\rfloor+1}|H_{k}(\omega)|}.\] Thus $\lim_{T\to\infty}T^{-1}\log Z_{T}(\omega)=0$ almost surely. The bounds relating $Z_{T}(\omega)$ and $n_{T}(\omega)$ give \[\lim_{T\to\infty}\frac{1}{T}\log n_{T}(\omega)=\beta_{0}\] for $\almost$, and therefore the same limit \[\lim_{T\to\infty}\frac{1}{T}\log P_{\psi,T}(f,\omega,\varphi,\epsilon)=\beta_{0}\] for $\almost$. The elementary bound $1\leq n_{T}(\omega)\leq\sum_{k=1}^{\lfloor T\rfloor+1}2^{k}<2^{\lfloor T\rfloor+2}$ gives a uniform bound on the normalized logarithms for $T\geq3$ at each fixed $\epsilon$. The dominated convergence theorem therefore yields \[\lim_{T\to\infty}\frac{1}{T}\int_{\Omega}\log P_{\psi,T}(f,\omega,\varphi,\epsilon)\,\probinte=\beta_{0}.\] Letting $\epsilon\to0$, we obtain $\pres_{\psi}(f,\varphi)=\beta_{0}$.

An induced equilibrium state can also be identified explicitly. The measure $\mu:=\prob\otimes\nu_{q}$ is invariant and ergodic. Its relative entropy and scaling integral are $h_{\mu}(f)=-q\log q-(1-q)\log(1-q)$, and \[\int_{J}\psi\,\dd\mu=1+\int_{\Omega}\omega_{0}\,\probinte+\int_{X}x_{0}\,\dd\nu_{q}=\frac{3}{2}+q.\] Since $\log(1+\e^{-\beta_{0}})=3\beta_{0}/2$, we have \[\frac{h_{\mu}(f)+\int_{J}\varphi\,\dd\mu}{\int_{J}\psi\,\dd\mu}=\beta_{0}=\pres_{\psi}(f,\varphi).\] Hence $\mu$ is an ergodic induced equilibrium state for $(0,\psi)$.

\section*{Acknowledgments}
The author would like to thank J. Jaerisch for valuable discussions and comments. This work was financially supported by JST SPRING, Grant Number JPMJSP2125. The author would like to thank the ``THERS Make New Standards Program for the Next Generation Researchers''. This work was supported by the Research Institute for Mathematical Sciences, an International Joint Usage/Research Center located in Kyoto University.

\bibliography{reference}
\bibliographystyle{alpha}

\end{document}